\documentclass[10pt,reqno]{amsart}
\usepackage{amsmath,amscd,amssymb}
\usepackage{url}
\usepackage{color}

\theoremstyle{plain}
   \newtheorem{theorem}{Theorem}[section]
   \newtheorem{proposition}[theorem]{Proposition}
   \newtheorem{lemma}[theorem]{Lemma}
   \newtheorem{corollary}[theorem]{Corollary}
   \newtheorem{conjecture}[theorem]{Conjecture}
   
\theoremstyle{definition}

   \newtheorem{example}{Example}[section] 
\theoremstyle{remark}
 \newtheorem{remark}{Remark}[section]

\numberwithin{equation}{section}
\newcommand{\R}{\mathbb{R}}
\newcommand{\C}{\mathbb{C}}
\newcommand{\Z}{\mathbb{Z}}
\newcommand{\B}{\mathcal{B}}

\newcommand{\E}{\mathcal{E}}
\newcommand{\A}{\mathcal{A}}
\newcommand{\PP}{\mathcal{P}}
\newcommand{\Zone}{\mathcal{Z}}
\newcommand{\D}{\mathcal{D}}
\newcommand{\HH}{\mathcal{H}}
\newcommand{\I}{\mathcal{I}}
\newcommand{\RR}{\mathcal{R}}

\def\newop#1{\expandafter\def\csname #1\endcsname{\mathop{\rm
#1}\nolimits}}

\newop{vol}
\newop{Vol}
\newop{conv}
\newop{Supp}
\newop{supp}
\newop{aw-height}
\newop{Nasc}
\newop{nasc}
\newop{comp}
\newop{Des}
\newop{des}
\newop{maxDes}
\newop{awheight}
\newop{lcm}
\newop{red}
\newop{sd}
\newop{exc}
\newop{axc}
\newop{Re}
\newop{Im}
\newop{IP}

\begin{document}

\title{Symmetric decompositions and real-rootedness}

\author{Petter Br\"and\'en}
\date{\today}
\address{Matematik, KTH, SE-100 44 Stockholm, Sweden}
\email{pbranden@kth.se}

\author{Liam Solus}
\date{\today}
\address{Matematik, KTH, SE-100 44 Stockholm, Sweden}
\email{solus@kth.se}

\begin{abstract}
In algebraic, topological, and geometric combinatorics inequalities among the coefficients of combinatorial polynomials are frequently studied.  
Recently a notion called the alternatingly increasing property, which is stronger than unimodality, was introduced. 
In this paper, we relate the alternatingly increasing property to real-rootedness of the symmetric decomposition of a polynomial to develop a systematic approach for proving 
the alternatingly increasing property for several classes of polynomials. 

We apply our results to strengthen and generalize real-rootedness, unimodality, and alternatingly increasing results pertaining to colored Eulerian and derangement polynomials, Ehrhart $h^\ast$-polynomials for lattice zonotopes, $h$-polynomials of barycentric subdivisions of doubly Cohen-Macaulay level simplicial complexes, and certain local $h$-polynomials for subdivisions of simplices. 
In particular, we prove two conjectures of Athanasiadis.
\end{abstract}

\maketitle
\thispagestyle{empty}

%---SECTION: Introduction------------
\section{Introduction}
\label{sec: introduction}
A long-standing endeavor in algebraic, topological, and geometric  combinatorics is to understand the distributional properties of sequences $p_0,\ldots,p_d$ of nonnegative numbers carrying combinatorial, algebraic, topological and/or geometric information, see e.g. \cite{B16,B15,B89,B94b,S89}. Such properties are frequently phrased equivalently in terms of its \emph{generating polynomial} $p :=p_0+p_1x+\cdots+p_dx^d$.
The generating polynomial $p$ is called \emph{unimodal} if there exists an index $t\in[d]:=\{1,\ldots,d\}$ such that
$
p_0\leq p_1\leq \cdots \leq p_t\geq \cdots \geq p_{d-1}\geq p_d,  
$
%It is called \emph{log-concave} if $p_k^2\geq p_{k-1}p_{k+1}$ for all indices $k\in[d-1]$, 
and it is called \emph{symmetric} (with respect to $d$) if $p_k = p_{d-k}$ for all indices $k=0,1,\ldots,d$.  
We say that $p$ is \emph{real-rooted} if all of its zeros are real numbers, or $p \equiv 0$. 
If $p$ is a real-rooted polynomial with only nonnegative coefficients then it is log-concave and unimodal \cite[Theorem 1.2.1]{B89}.  
%In particular, one way to show a combinatorial sequence is unimodal is to show that its generating polynomial is real-rooted.  

If the coefficients of $p$ satisfy
\[
0\leq p_0\leq p_d\leq p_1\leq p_{d-1}\leq \cdots\leq p_{\left\lfloor\frac{d+1}{2}\right\rfloor},
\]
we say that $p$ is \emph{alternatingly increasing}. 
This property, which is stronger than unimodality, has recently been the focus of a variety of conjectures in combinatorics as well as a means by which to prove unimodality \cite{A14,A17,AS13,BJM16,SV13,S17}.  
The alternatingly increasing property of $p$ is inherently tied to a unique \emph{symmetric decomposition} of $p$: 
Every polynomial $p$ of degree $d$ can be uniquely decomposed as $p = a+xb$ where $a$ and $b$ are symmetric with respect to $d$ and $d-1$, respectively.  
Moreover, a generating polynomial $p$ is alternatingly increasing if and only if $a$ and $b$ are both unimodal and have nonnegative coefficients.  
One purpose of this paper is to study when the unimodality of the polynomials $a$ and $b$ can be recovered by showing that they are also real-rooted.  
Going further, we also relate real-rootedness of $a$ and $b$ to that of $p$ via the theory of interlacing polynomials.

The remainder of this paper is organized as follows: 
In Section~\ref{sec: the real-rootedness properties of symmetric decompositions} we prove sufficient conditions for a polynomial $p$ to have real-rooted $a$- and $b$-polynomials in its symmetric decomposition.  
We also provide a characterization of when $b$ interlaces $a$, a stronger condition that implies real-rootedness of both polynomials.  
In the remaining sections, we apply the results of Section~\ref{sec: the real-rootedness properties of symmetric decompositions} to some open questions in the combinatorial literature. 
In Section~\ref{sec: colored}, we show that the colored Eulerian and derangement polynomials all have symmetric decompositions for which $b$ interlaces $a$.  
This strengthens the previously known real-rootedness results.  
Additionally, these results provide an affirmative answer to some questions posed by Athanasiadis \cite{A16,A17} on the real-rootedness of certain local $h$-polynomials.  
In Section~\ref{sec: h-polynomials of lattice zonotopes}, we show that the $h^\ast$-polynomial of a lattice zonotope containing an interior lattice point with respect to any combinatorially positive valuation has real-rooted $a$- and $b$-polynomials.  
In the case that the lattice zonotope is centrally-symmetric, we further prove that these $h^\ast$-polynomials have $b$ interlacing $a$. 
These results generalize and strengthen those of \cite{BJM16,SV13}.   
In Section~\ref{sec: h-polynomials of cohen-macaulay simplicial complexes}, we show that the $h$-polynomial of the barycentric subdivision of any doubly Cohen-Macaulay level simplicial complex has real-rooted $a$- and $b$-polynomials, thereby strengthening some results of \cite{BW09}.

%---SECTION: The Real-rootedness Properties of Symmetric Decompositions------------
\section{Real-rootedness and Symmetric Decompositions}
\label{sec: the real-rootedness properties of symmetric decompositions}
For a polynomial $p\in\R[x]$ with degree at most $d$, we let 
\[
\I_d(p) = x^dp(1/x),  
\]
or, when $d$ is understood, we may simply write $\I(p)$.  
When $d$ is the degree of $p$, the polynomial $\I(p)$ is sometimes called the \emph{reciprocal} of $p$. 
The following lemma is an exercise in linear algebra (see for example \cite[Exercise 10.13]{BR07}).

%---LEMMA: Symmetric I-decomposition----
\begin{lemma}
\label{lem: symmetric I-decomposition}
Let $p\in\R[x]$ be of degree at most $d$.  
Then there exist unique polynomials $a,b\in\R[x]$ such that
\[
p = a+xb,
\]
$\deg(a)\leq d$, $\deg(b)\leq d-1$, $\I_d(a) = a$, and $\I_{d-1}(b) = b$.  
Moreover, 
\[
a = \frac {p-x\I_d(p)}{1-x} 
\quad
\mbox{and }
\quad
b= \frac {\I_d(p)-p}{1-x}.
\]
\end{lemma}

We will call the ordered pair of polynomials $(a,b)$ from Lemma~\ref{lem: symmetric I-decomposition} the \emph{(symmetric) $\I_d$--decomposition} of the polynomial $p$.  
We will say that the $\I$--decomposition $(a,b)$ of $p$ is \emph{real-rooted} whenever both $a$ and $b$ are real-rooted. 
 
Many questions about the distributional properties of polynomials in combinatorics pertain to various types of $h$-\emph{polynomials}; i.e., a polynomial $h\in\R[x]$ satisfying the relation
\begin{equation}
\label{eqn:1}
\sum_{m\in\Z_{\geq0}}i(m)x^m = \frac{h(x)}{(1-x)^{d+1}},
\end{equation}
with respect to some polynomial $i\in\R[x]$ of degree $d$.  
Oftentimes, the above series is regarded as the Hilbert series of a finitely generated graded algebra.  
Closely tied to an $h$-polynomial is its $f$-\emph{polynomial} which is given via the transformation
\begin{equation}
\label{eqn: f-polynomial}
f(h;x) := (1+x)^dh\left(\frac{x}{1+x}\right).
\end{equation}
Given a polynomial $h\in\R[x]$ of degree at most $d$, we will call the unique polynomial $f$ produced from $h$ via the transformation in \eqref{eqn: f-polynomial} the $f$-polynomial of $h$ (with respect to $d$).  
When the polynomial $h$ is understood from context, we will simply write $f$ for its $f$-polynomial.
Notice that if $h$ has only nonnegative coefficients and real-zeros, then $f$ has only nonnegative coefficients and zeros in the interval $[-1,0]$.  
When studying the zeros of an $h$-polynomial, it is sometimes easier to work with the zeros of its polynomial $f$.  
This motivates the following definition and lemma:
For $p\in\R[x]$ of degree at most $d$ we let
\begin{equation}
\label{eqn: S-transformation}
\RR_d(p)(x) = (-1)^dp(-1-x),
\end{equation}
or when $d$ is understood we simply write $\RR(p)$.  
The proof of the following lemma follows from considering the $f$-polynomial of the polynomial $p$ in Lemma~\ref{lem: symmetric I-decomposition}.  
%---LEMMA: Symmetric S-decomposition----
\begin{lemma}
\label{lem: symmetric S-decomposition}
Let $p\in\R[x]$ be of degree at most $d$.  
Then there exist unique polynomials $\tilde{a},\tilde{b}\in\R[x]$ such that
\[
p = \tilde{a}+x\tilde{b},
\]
$\deg(\tilde{a})\leq d$, $\deg(\tilde{b})\leq d-1$, $\RR_d(\tilde{a}) = \tilde{a}$, and $\RR_{d-1}(\tilde{b}) = \tilde{b}$.  
Moreover, 
\[
\tilde{a}=(1+x)p-x\RR_d(p)
\quad
\mbox{and }
\quad
\tilde{b}=\RR_d(p)-p.
\]
Additionally, $f(\I_d(p);x) = \RR_d(f(p;x))$, where the $f$-polynomials are computed with respect to degree $d$.
\end{lemma}
Analogous to Lemma~\ref{lem: symmetric I-decomposition}, we call the ordered pair of polynomials $(\tilde{a},\tilde{b})$ from Lemma~\ref{lem: symmetric S-decomposition} the \emph{(symmetric) $\RR_d$--decomposition} of the polynomial $p$, and we say that the $\RR_d$-decomposition of $p$ is real-rooted whenever both $\tilde{a}$ and $\tilde{b}$ are real-rooted.  
The exact relationship between the $\I$-decomposition of an $h$-polynomial and the $\RR$-decomposition of its $f$-polynomial is summarized by the following lemma.
%---LEMMA: Relationship between I and S decompositions----
\begin{lemma}
\label{lem: relationship between I and S decompositions}
Suppose $h$ is a polynomial of degree at most $d$, and let $f$ denote its $f$-polynomial with respect to $d$.  
If $(a,b)$ is the $\I_d$-decomposition of $h$ and $(\tilde{a},\tilde{b})$ is the $\RR_d$-decomposition of $f$, then  
\[
\tilde{a}= (1+x)^da(x/(1+x)) \ \ \mbox{ and } \ \ \tilde{b}= (1+x)^{d-1}b(x/(1+x)).
\]
That is, $\tilde{a}$ is the $f$-polynomial of $a$, and $\tilde{b}$ is the $f$-polynomial of $b$ (w.r.t.~$d$ and~$d-1$, respectively).  
\end{lemma}

%---EXAMPLE: An example of the previous lemmas-----
\begin{example}
\label{ex: lemmas}
We now give an example of the polynomials and their various decompositions discussed in the previous lemmas.
Suppose that 
\[
h = 1+1018x+10678x^2+14498x^3+2933x^4+32x^5.
\]
Its $f$-polynomial with respect to $d = 5$ is then
\[
f = 1+1023x+14760x^2+52650x^3+68040x^4+29160x^5.
\]
The $\I_5$-decomposition $(a,b)$ of $h$ is then
\begin{equation*}
\begin{split}
a &= 1+987x+12814x^2+12814x^3+987x^4+x^5,\\
b &= 31+1946x+39836x^2+1946x^3+31x^4,\\
\end{split}
\end{equation*}
and the $\RR_5$-decomposition $(\tilde{a},\tilde{b})$ of $f$ is 
\begin{equation*}
\begin{split}
\tilde{a} &= 1+992x+12690x^2+40860x^3+48600x^4+19440x^5,\\
\tilde{b} &= 31+2070x+11790x^2+19440x^3+9720x^4.\\
\end{split}
\end{equation*}
Using the above expressions, one can check the various relationships between these polynomials described in Lemmas~\ref{lem: symmetric I-decomposition},~\ref{lem: symmetric S-decomposition}, and~\ref{lem: relationship between I and S decompositions}. 
\end{example}

%---SUBSECTION: Interlacing Properties for Symmetric Decompositions-----
\subsection{Interlacing properties for symmetric decompositions}
\label{subsec: interlacing properties for symmetric decompositions}
Recall that the zeros of two real-rooted polynomials $p$ and $q$  in $\R[x]$ \emph{interlace} if there is a zero of $p$ between each pair of zeros (counted with multiplicity) of $q$ (and vice versa). 
If $p$ and $q$ have interlacing zeros, then the \emph{Wronskian} $W[p,q]:=p'q-pq'$ is either nonpositive on $\R$ or nonnegative on $\R$. 
We write $p \prec q$ if $p$ and $q$ are real-rooted, their zeros interlace, and $W[p,q] \leq 0$ on all of $\R$. 
For technical reasons we consider the identically zero polynomial to be real-rooted and write $0 \prec p$ and $p \prec 0$ for any other real-rooted polynomial $p$. 
%---REMARK: Relating Notions of Interlacing-----
\begin{remark}
If the signs of the leading coefficients of two real-rooted polynomials $p$ and $q$ in $\R[x]$ are positive, then $p \prec q$ if and only if 
$$
  \cdots \leq \beta_2 \leq \alpha_2 \leq \beta_1\leq \alpha_1,
$$
where $\cdots \leq \beta_2 \leq \beta_1$ and $\cdots \leq \alpha_2 \leq \alpha_1$ are the zeros of $p$ and $q$, respectively. Indeed, by continuity, we may assume that $p$ and $q$ have no common zeros. 
The correct sign of the Wronskian is read off when it is evaluated at the largest zero of $pq$. 
This translates between the notions of interlacing in \cite{B06} and the notion in this paper.
\end{remark}

A polynomial $f \in \C[x]$ is called \emph{stable} if either $f \equiv 0$ or all the zeros of $f$ have nonpositive imaginary parts. 
The Hermite-Biehler theorem \cite[Theorem 6.3.4]{RG02} relates stability to interlacing zeros. 
%It essentially says that the zeros of $p$ and $q$ interlace if and only if $p+iq$ is stable.  
%We say that $q$ \emph{interlaces} $p$, denoted $q\prec p$, if and only if the polynomial $p+iq\in\C[x]$ is stable. 
\begin{theorem}[Hermite-Biehler Theorem]
\label{thm: Hermite-Biehler}
Let $p, q \in \R[x]$. Then $p \prec q$ if and only if $q+ip$ is stable.  
\end{theorem}
For a proof of Theorem~\ref{thm: Hermite-Biehler}, we refer to \cite{W11}\footnote{Note that there is a typo in the definition of proper position in \cite{W11}.  It is essential to add the condition that the zeros of $f$ and $g$ interlace.}.
One consequence of Theorem~\ref{thm: Hermite-Biehler} is that $p$ and $q$ in $\R[x]$ are real-rooted whenever $q+ip$ is stable.  
Moreover, it follows that for real-rooted polynomials $p$ and $q$ in $\R[x]$:
\begin{itemize}
\item $p \prec \alpha p$ for all $\alpha \in \R$,
\item $p\prec q$ if and only if $-q \prec p$,
\item $p\prec q$ if and only if $\alpha p \prec \alpha q$, for any $\alpha \in \R \setminus \{0\}$. 
\end{itemize}
We will also make use of the following proposition:
\begin{proposition}[Lemma 2.6 in \cite{BB}]\label{cones}
Let $p$ be a real-rooted polynomial that is not identically zero. The sets 
$$
\{q \in \R[x] : q \prec p\} \ \ \mbox{ and }  \ \ \{q \in \R[x] : p \prec q\}
$$
are convex cones. 
\end{proposition}
In a recent paper \cite{S17}, the second author showed that $b\prec a$ for the symmetric $\I$-decompositions of a family of Ehrhart $h^\ast$-polynomials in order to recover the alternatingly increasing property, and in fact real-rootedness.  
This interlacing condition turns out to have several equivalent forms, which we will use in this paper.  

%---THEOREM: Interlacing TFAE-----
\begin{theorem}
\label{thm: interlacing TFAE}
Let $p\in\R[x]$ have degree at most $d$ and $\I_d$--decomposition $(a,b)$, for which both $a$ and $b$ have only nonnegative coefficients.  
Then the following are equivalent:
	\begin{enumerate}
		\item $b\prec a$,
		\item $a\prec p$,
		\item $b\prec p$, 
		\item $\I_d(p) \prec p$.
		\item $\RR_d(f)\prec f$, where $f$ is the $f$-polynomial of $p$.  
	\end{enumerate}
\end{theorem}

\begin{proof}
We prove (1) $\Rightarrow$ (2) $\Rightarrow$ (3) $\Rightarrow$ (4) $\Rightarrow$ (1). The equivalence of (4) and (5) follows from \eqref{eqn: f-polynomial} and the fact that $f(\I_d(p)) = \RR_d(f(p))$ (see Lemma~\ref{lem: symmetric S-decomposition}).  

To see that (1) implies (2), notice that $b\prec a$ implies $a\prec xb$ since $a$ and $b$ have only nonpositive zeros and positive leading coefficients.  
Since $a\prec a$, it follows from Proposition \ref{cones} that $a\prec a+xb = p$.  

To see (2) implies (3), notice first that if $a\prec p$, then $p\prec -a$.  
Thus, since $p\prec p$, it follows that $p\prec p-a = xb$ by Proposition \ref{cones}.  
Equivalently, $b\prec p$.  

Assume (3), i.e., $b\prec p$. Then $p\prec xb$, and so $-xb\prec p$.  
Since $p\prec p$, it follows that $a = p-xb\prec p$ by Proposition \ref{cones}.
Therefore, $\I_d(p) = a+b \prec p$ by Proposition \ref{cones}, i.e., (4) holds.  

We now show that (4) implies (1).  
Assume  $\I_d(p)\prec p$.  
Since $p$ and $\I_d(p)$ have only nonpositive zeros, it follows that $p\prec x\I_d(p)$.  
Since $p\prec -p$, by Proposition~\ref{cones}, we have $p\prec x\I_d(p)-p=(x-1)a$.
Since $p$ has only nonpositive zeros, it follows that $a \prec p$.  
Since $p = a+xb$ and $a\prec -a$, by Proposition~\ref{cones}, we have that $a \prec xb$, and thus  $b\prec a$.  
%Assume $\I_d(p)\prec p$. 
%Notice that $\I_d(p)\prec \I_d(p)$ implies that $\I_d(p)\prec -x\I_d(p)$.  
%Additionally, since $\I_d(p)\prec p$, we have $-p\prec -x\I_d(p)$.  
%Thus, $\I_d(p)-p\prec -x\I_d(p)$ by Proposition \ref{cones}.  
%Similarly, since $\I_d(p) \prec p$ and $-p \prec p$, we have $\I_d(p) -p \prec p$ by Proposition \ref{cones}.  
%By Proposition \ref{cones} again we deduce $\I_d(p)-p \prec p-x\I_d(p)$, or equivalently $b \prec a$. 
\end{proof}

%---REMARK: Alternatingly Increasing Connection----
\begin{remark}
\label{rmk: alternatingly increasing connection}
Recall that, by definition, $p\prec q$ implies that the polynomials $p$ and $q$ have interlacing and only real zeros.  
It follows that the polynomials satisfying the conditions of Theorem~\ref{thm: interlacing TFAE} are all real-rooted with nonnegative coefficients, and hence log-concave and unimodal.  
Consequently, statement (1) implies that the polynomial $p$ has the alternatingly increasing property.  
Hence, proving anyone of these equivalent statements for a polynomial $p$ is sufficient to verify that $p$ is alternatingly increasing, although the conditions of Theorem~\ref{thm: interlacing TFAE} are much stronger than this property. 
In the coming sections, we will typically prove that statement (4) holds in order to deduce results on the alternatingly increasing property and real-rootedness of $\I$-decompositions. 
\end{remark}

Next we will prove some lemmas and sufficient conditions for a polynomial to have the interlacing properties given in Theorem~\ref{thm: interlacing TFAE}.  
To do so, we will make use of a linear transformation which is sometimes called the subdivision operator \cite{B15}.  
The \emph{subdivision operator} $\E : \R[x]\longrightarrow \R[x]$ is defined by
\[
\E \binom x k = x^k, \ \ \mbox{ for all } k \geq 0,  
\]
where $\binom x k = x(x-1)\cdots (x-k+1)/k!$. 
An important property of the subdivision operator is that it maps the polynomial $i$ in \eqref{eqn:1} to the $f$-polynomial its associated $h$-polynomial.  
%---LEMMA: Subdivision Operator----
\begin{lemma}
\label{lem: subdivision operator}
Suppose that $i\in\R[x]$ is a polynomial of degree $d$ and let
\[
\sum_{m\geq0}i(m)x^m = \frac{h}{(1-x)^{d+1}}.
\]
Then 
$
\E(i)(x) = f(h;x)
$
with respect to degree $d$.
\end{lemma}

\begin{proof}
Note that, by using the well-known identity
\[
\sum_{m\geq0}{m\choose k}y^m = \frac{y^k}{(1-y)^{k+1}}
\]
and setting $y = x/(1+x)$, one can deduce
\[
\E(i) = \sum_{m\geq0}i(m)\frac{x^m}{(1+x)^{m+1}}.
\]
Since, the right-hand-side of this equation is equal to $(1+x)^dh(x/(1+x))$, the result follows. 
\end{proof}

The subdivision operator has an associated product map, in regards to which the $\RR$-transformation given in equation~\eqref{eqn: S-transformation} is multiplicative.
The \emph{diamond product} of two polynomials $p,q\in\R[x]$ is defined by 
\[
p \diamond q(x) := \E\left( \E^{-1}(p)\cdot \E^{-1}(q)\right)(x)= \sum_{k\geq 0} \frac {p^{(k)}(x)} {k!} \frac {q^{(k)}(x)} {k!}x^k(x+1)^k. 
\]
For the second equality, see \cite{W92}. 
It follows from \cite[Lemma 4.3]{B06} that $\RR\circ \E = \E \circ \RR$, and consequently, 
\[
\RR(p\diamond q) = \RR(p)\diamond \RR(q).
\] 
In the following, we let $E_k^d := \E\left(x^k(x+1)^{d-k}\right)$, $0 \leq k \leq d$. 
Note that $\RR(E_k^d)=E_{d-k}^d$. 
Further let
\[
\E_d := \left\{ p= \sum_{i=0}^d c_i E_i^d : c_i \geq 0 \mbox{ for all } i \right\},
\quad
\mbox{and}
\quad 
\A_d := \left\{ p \in \E_d :  \RR_d(p) \prec p\right\}.
\]
By a theorem due to the first author we know that each $p \in \E_d$ is real-rooted:
%----PROPOSITION: E-map on [-1,0]-----
\begin{proposition}
[Theorem 4.2 in \cite{B06}]
\label{fundE}
Each $p \in \E_d$ has all its zeros in the interval $[-1,0]$. 
Moreover $E_0^d\prec p\prec E_d^d$.
\end{proposition}
The next proposition says that $\prec$ is a partial order on (monic polynomials in) $\E_d$. 
\begin{proposition}\label{poset}
The interlacing relation $\prec$ is a partial order on monic polynomials in $\E_d$, with minimal element $E_0^d/d!$ and maximal element $E_d^d/d!$. 
\end{proposition}
\begin{proof}
We need to prove transitivity. 
If $f, g, h \in \E_d$, $f \prec g$ and $g \prec h$, then 
$$
E_0^d \prec f \prec g \prec h \prec E_d^d
$$
by Proposition \ref{fundE}. Moreover, since $(x+1)E_d^d = xE_0^d$ (\cite[Lemma 4.5]{B06}) we have 
$E_0^d \prec E_d^d$. The proof now follows from \cite[Lemma 2.3]{B06}. 
\end{proof}
The main goal of the coming results is to understand when polynomials in $\E_d$ are also in $\A_d$.  
The next lemma proves that $\tilde{a}$ in the $\RR$-decomposition of a polynomial $f \in \E_d$ is real-rooted; a condition which is necessary for $f$ to satisfy the interlacing properties in Theorem~\ref{thm: interlacing TFAE} by way of Lemma~\ref{lem: relationship between I and S decompositions}.
%---LEMMA: Real-Rooted a Polynomial---
\begin{lemma}
\label{lem: real-rooted a polynomial}
Let $p= \sum_{k=0}^d a_kx^k(x+1)^{d-k}$ where $a_k \geq 0$ for all $k$. 
If $(\tilde{a},\tilde{b})$ is the $\RR$-decomposition of $\E(p)$, then $\tilde{a}$ is real-rooted and $\tilde{a} \prec (x+1)E_d^d$.
\end{lemma}

\begin{proof}
Let $q = \E(p)$.  Then $(x+1)q \prec (x+1)E_d^d$ by Proposition \ref{fundE}, and 
\[
(x+1)E_d^d = xE_0^d\prec x\RR(q),
\]
where the equality follows from \cite[Lemma 4.5]{B06}, and the interlacing statement follows from Proposition \ref{fundE} since $\RR(q) \in \E_d$.  
This implies  $-x\RR(q)\prec(x+1)E_d^d$. 
%Moreover $(x+1)q \prec (x+1)E_d^d$, by Proposition \ref{fundE}. 
Therefore, 
\[
\tilde{a}=(x+1)q-x\RR(q)\prec(x+1)E_d^d,
\]
by Proposition \ref{cones}, which completes the proof.
\end{proof}

Analogously, by way of Lemma~\ref{lem: symmetric S-decomposition}, the following lemma provides sufficient conditions for the polynomial $\tilde{b}$ in the $\RR$-decomposition of a polynomial $f$ to be real-rooted.  
%---LEMMA: Pair Sums of E_i^d-----
\begin{lemma}
\label{lem: pair sums of E_i^d}
If $k+\ell \geq d$, then $p=E_k^d + E_\ell^d \in \A_d$. 
Moreover,
\[
E_{0}^{d-1} \prec \RR(p)-p \prec E_{d-1}^{d-1}.
\]
\end{lemma}

\begin{proof}
The proof is by induction on the degree $d$, with the base case of $d=1$ being trivial.
Now suppose the result is true up to $d-1$. 
If $k+\ell=d$, then $\RR(E_k^d + E_\ell^d)= E_k^d + E_\ell^d$, and we are done. 
Suppose $k+\ell \geq d+1$. Then $k, \ell \geq 1$ and 
\begin{align*}
E_k^d + E_\ell^d
&= \E\left( x \cdot \left( x^{k-1}(x+1)^{d-1-(k-1)} + x^{\ell-1}(x+1)^{d-1-(\ell-1)}\right)\right), 	\\
&= x \diamond \left(E_{k-1}^{d-1} + E_{\ell-1}^{d-1} \right).	
\end{align*}
Since the diamond products preserves interlacing in both coordinates \cite{B04}, and since $k-1+\ell-1 \geq d-1$, we have by induction
\begin{align*}
\RR(E_k^d + E_\ell^d)	= (x+1) \diamond \RR(E_{k-1}^{d-1} + E_{\ell-1}^{d-1})
&\prec (x+1) \diamond  \left(E_{k-1}^{d-1} + E_{\ell-1}^{d-1} \right) \\
& \prec x\diamond  \left(E_{k-1}^{d-1} + E_{\ell-1}^{d-1} \right) = E_k^d + E_\ell^d.	
\end{align*}
The first part of the lemma now follows from Proposition \ref{poset}. 

Let $q= E_{k-1}^{d-1}+ E_{\ell-1}^{d-1}$. 
Then $p=x\diamond q$, as observed above, and 
\[
\RR(p)-p= (x+1)\diamond \RR(q)-x \diamond q = x\diamond (\RR(q)-q) +\RR(q).
\]
By induction $\RR(q)-q \prec E_{d-2}^{d-2}$, and hence $x \diamond (\RR(q)-q) \prec x\diamond E_{d-2}^{d-2}=E_{d-1}^{d-1}$, since $\diamond$ preserves interlacing \cite{B04}. 
Also, since $\RR(q) \prec E_{d-1}^{d-1}$, we have $\RR(p)-p \prec E_{d-1}^{d-1}$ by Proposition \ref{cones}. 
Similarly 
\[
\RR(p)-p= (x+1) \diamond (\RR(q)-q)+q.
\]
So again by induction, $E_0^{d-2} \prec \RR(q)-q$, and hence 
\[
E_0^{d-1} = (x+1)\diamond E_0^{d-2} \prec (x+1)\diamond (\RR(q)-q).
\]
The proof now follows from Proposition \ref{cones}, since $E_0^{d-1} \prec q$. 
\end{proof}

The remaining lemma in this subsection says that $\A_d$ is a filter in $\E_d$ (under the partial order $\prec$).  
This result will be applied to families of generating polynomials in the coming sections of this paper. 
%---LEMMA: Some First Properties of A_d----
\begin{lemma}
\label{lem: first properties of A_d}
If $p \in \A_d, q \in \E_d$, and $p \prec q$, then $q \in \A_d$.
\end{lemma}

\begin{proof}
To prove the claim, we will rely on the following fact: If $p$ and $q$ are two degree $d$ polynomials and $p \prec q$, then $\RR_d(q)\prec \RR_d(p)$.  Hence if $p \prec q$  and $\RR_d(p) \prec p$, then 
\[
\RR_d(q) \prec \RR_d(p) \prec p \prec q.
\]
The proof now follows from Proposition \ref{poset}.  
\end{proof}

%---SUBSECTION: More Real-rooted Symmetric Decompositions----
\subsection{More real-rooted symmetric decompositions}
\label{subsec: more real-rooted symmetric decompositions}
A natural relaxation of Theorem~\ref{thm: interlacing TFAE} is to simply ask when the $\I$-decomposition $(a,b)$ of a polynomial $p$ is real-rooted, without regard to whether or not $b\prec a$.
In the context of algebraic combinatorics, we are particularly interested in this question when $p$ is regarded as an $h$-polynomial.    
Thus, our goal in this subsection is to prove a sufficient condition for an $h$-polynomial to have a real-rooted $\I$--decomposition.  
Specifically, we prove the following theorem:

%---THEOREM: Real-Rooted I-Decomposition-----
\begin{theorem}
\label{thm: real-rooted I-decomposition}
Let 
\[
i= \sum_{j=0}^d c_j x^j(x+1)^{d-j},
\]
where $c_j \geq 0$ for all $0\leq j\leq d$, and 
\[
\sum_{m\geq0} i(m)x^m = \frac {h}{(1-x)^{d+1}}. 
\]
If 
\begin{equation}
\label{eqn: hibi}
c_0+c_1+\cdots+c_j \leq c_d+c_{d-1}+\cdots+c_{d-j}
\end{equation}
 for all $0\leq j \leq d/2$, then the $\I_d$-decomposition of $h$ is real-rooted.
\end{theorem}

\begin{proof}
Let $(a,b)$ be the $\I$--decomposition of $h$ given as in Theorem~\ref{thm: real-rooted I-decomposition}.
Also let $(\tilde{a},\tilde{b})$ be the $\RR$--decomposition of $i(x)$ and $(\tilde{A},\tilde{B})$ be the $\RR$--decomposition of $\E(i)$.  
Recall from Lemma~\ref{lem: subdivision operator} that $\E(i)$ and $h$ are related by 
\[
\E(i)(x) = (1+x)^{d}h\left(\frac{x}{1+x}\right).
\]
By Lemma~\ref{lem: relationship between I and S decompositions}, we also have that
\[
\tilde{A}(x)= (1+x)^da\left(\frac{x}{1+x}\right)
\quad 
\mbox{ and } 
\quad 
\tilde{B}(x)= (1+x)^{d-1}b\left(\frac{x}{1+x}\right).
\]
Thus, the real-rootedness of $a$ then follows from the real-rootedness of $\tilde{A}$ as guaranteed by Lemma~\ref{lem: real-rooted a polynomial}.  
As for the $b$ polynomial, if we know that 
\[
c_0+c_1+\cdots+c_j \leq c_d+c_{d-1}+\cdots+c_{d-j}
\]
for all $0\leq j \leq d/2$, we can apply Lemma~\ref{lem: pair sums of E_i^d}.  
These inequalities imply that $\E(i)$ is a nonnegative combination of the polynomials $E_k^d+E_j^d$ for $k+j\geq d$, from which it follows that $E_{0}^{d-1} \prec \RR\E(i)-\E(i) \prec E_{d-1}^{d-1}$ by Proposition \ref{cones}. 
By Lemma~\ref{lem: symmetric S-decomposition}, we know that $\tilde{B} = \RR\E(i)-\E(i)$, and therefore conclude that $b$ must be real-rooted as well.
\end{proof}

%--REMARK: On Applications----
\begin{remark}
[On applications of Theorem~\ref{thm: real-rooted I-decomposition}]
\label{rem: on applications}
The inequalities~\eqref{eqn: hibi} assumed to hold for the coefficient sequence $c_0,\ldots,c_d$ in Theorem~\ref{thm: real-rooted I-decomposition} appear frequently in algebra and combinatorics.  
Any pure $O$-sequence is known to satisfy the inequalities $c_j\leq c_{d-j}$ for all $0\leq j\leq d$  \cite{H89}, and these together imply the inequalities~\eqref{eqn: hibi}.
Moreover, the inequalities~\eqref{eqn: hibi} are specifically known to hold for the $h$-vector of any semistandard graded Cohen-Macaulay domain, so in particular, for any Ehrhart $h^\ast$-polynomial \cite{S91}.   
Theorem~\ref{thm: real-rooted I-decomposition} provides a means by which to produce an $h$-polynomial with a real-rooted $\I$--decomposition from any such sequence.  
The interesting applications arise when the resulting $h$-polynomial has known combinatorial interpretations, as we will see in the coming sections.  
\end{remark}

%---SECTION: Colored Eulerian and Derangement Polynomials-----
\section{Colored Eulerian and Derangement Polynomials}
\label{sec: colored}
For positive integers $n$ and $r$ we say that a pair $(\pi,\emph{z})\in\mathfrak{S}_n\times \{0,1,\ldots,r-1\}^n$ is an $r$-\emph{colored permutation} of $[n]$, and we denote the collection of all $r$-colored permutations by $\Z_r\wr\mathfrak{S}_n$.  
Given $(\pi,\emph{z})\in\Z_r\wr\mathfrak{S}_n$, we think of $z_k$ in the $n$-tuple $\emph{z}$ as the color of $\pi_k$ in the permutation $\pi$. %and we may denote $(\pi,\emph{z})$ by $\pi_1^{z_{1}}\pi_2^{z_{2}}\cdots\pi_n^{z_{n}}$.  
%Following Steingr\'imsson \cite{S94}, we say that $k\in[p]$ is a \emph{descent} of $(\pi,\emph{z})\in\Z_r\wr\mathfrak{S}_n$ if either $z_k>z_{k+1}$ or $z_k = z_{k+1}$ and $\pi_k>\pi_{k+1}$ where we assume $\pi_{n+1} = n+1$ and $z_{n+1} = 0$.  
Following Steingr\'imsson \cite{S94}, we say that $k\in[p]$ is an \emph{excedance} of $(\pi,\emph{z})\in\Z_r\wr\mathfrak{S}_n$ if either $\pi_k>k$ or $\pi_k = k$ and $z_k>0$.  
The number of excedances of $(\pi,\emph{z})\in\Z_r\wr\mathfrak{S}_n$ is denoted $\exc(\pi,\emph{z})$.  
The \emph{Eulerian polynomial} for the $r$-colored permutations of $[n]$ is
\[
A_{n,r}(x) :=\sum_{(\pi,\emph{z})\in\Z_r\wr\mathfrak{S}_n}x^{\exc(\pi,\emph{z})} %= \sum_{(\pi,\emph{z})\in\Z_r\wr\mathfrak{S}_n}x^{\des(\pi,\emph{z})},
\]
%where the second equality follows from \cite[Theorem 15]{S94}.
for $n\geq 1$, and we set $A_{0,r}:=1$ for all $r\geq1$. 
An $r$-colored permutation having no fixed points of color $0$ (i.e., no indices $k$ for which $\pi_k = k$ and $z_k = 0$) is called a \emph{derangement} and the collection of all derangements in $\Z_r\wr\mathfrak{S}_n$ is denoted by $\mathfrak{D}_{n,r}$.  
The \emph{derangement polynomial} for the $r$-colored permutations of $[n]$ is
\[
d_{n,r}(x) :=\sum_{(\pi,\emph{z})\in\mathfrak{D}_{n,r}}x^{\exc(\pi,\emph{z})}
\]
for $n\geq1$, and we set $d_{0,r} :=1$ for all $r\geq1$. 
When $r = 1$, we recover the classical  $n^{th}$ \emph{Eulerian polynomial} and  $n^{th}$ \emph{derangement polynomial}, which we will respectively denote by $A_n$ and $d_n$.  
The first few Eulerian and derangement polynomials, respectively, are listed in Table~\ref{tab:1} and Table~\ref{tab:2}.  
%---TABLE:  Eulerian Polynomials-----
\begin{table}
\caption{The Eulerian polynomials $A_n$ for $1\leq n\leq 6$.}
\label{tab:1}
\centering
%\begin{tabsize}
\begin{tabular}{ l | l }
%\toprule
$n$	&	$A_n$	\\\hline
%\midrule
$1$	&	$1$	\\
$2$	&	$1+x$	\\
$3$	&	$1+4x+x^2$	\\
$4$	&	$1+11x+11x^2+x^3$	\\
$5$	&	$1+26x+66x^2+26x^3+x^4$	\\
$6$	&	$1+57x+302x^2+302x^3+57x^4+x^5$	\\
%\bottomrule
\end{tabular}
%\end{tabsize}
\end{table}
%
%---TABLE:  Eulerian Polynomials-----
\begin{table}[b!]
\caption{The derangement polynomials $d_n$ for $1\leq n\leq 6$.}
\label{tab:2}
\centering
%\begin{tabsize}
\begin{tabular}{ l | l }
%\toprule
$n$	&	$d_n$	\\\hline
%\midrule
$1$	&	$0$	\\
$2$	&	$x$	\\
$3$	&	$x+x^2$	\\
$4$	&	$x+7x^2+x^3$	\\
$5$	&	$x+21x^2+21x^3+x^4$	\\
$6$	&	$x+51x^2+161x^3+51x^4+x^5$	\\
%\bottomrule
\end{tabular}
%\end{tabsize}
\end{table}
In \cite{S94}, it is shown that $A_{n,r}$ are real-rooted for all $n\geq 1$ and $r\geq 1$.  
In \cite{B93}, \cite{CTZ09,C09}, and \cite{CM10}, respectively, it is also noted that $d_n$, $d_{n,2}$, and $d_{n,r}$ are real-rooted for $n\geq1$ and $r\geq 1$.  
In this section, we strengthen these results by showing that $A_{n,r}$ and $d_{n,r}$ are all interlaced by their own reciprocal for all $n\geq1$ and $r\geq1$.
By Theorem~\ref{thm: interlacing TFAE}, it then follows that $A_{n,r}$ and $d_{n,r}$ always have real-rooted $I$-decompositions, and thus are alternatingly increasing for all $n\geq1$ and $r\geq 1$.  
Finally, we apply these results to answer a question on the real-rootedness of certain local $h$-polynomials.

%---SUBSECTION: The colored Eulerian polynomials------------
\subsection{The colored Eulerian polynomials}
\label{sec: the colored eulerian polynomials}
To prove the desired result, we first prove a slightly more general statement.
In the following, for $r\geq2$, $n\geq1$, and $0\leq k\leq n$ we let
\begin{equation}
\label{eqn: colored E}
E_{r,k}^n :=\E((rx)^k(rx+1)^{n-k}). 
\end{equation}
We also define the polynomials $A_{r,k}^n(x)$ by the relation
\begin{equation}
\label{eqn: partial eulerian}
\sum_{m\geq0}(rm)^k(rm+1)^{n-k}x^m = \frac{A_{r,k}^n(x)}{(1-x)^{n+1}}.
\end{equation}
Note that $A_{r,0}^n = A_{n,r}$ for all $n,r\geq 1$.  
We then have the following theorem.
%---THEOREM: Colored Eulerian Polynomials----
\begin{theorem}
\label{thm: colored eulerian polynomials}
Suppose that $p\in\R[x]$ is polynomial satisfying
\[
p(x) = \sum_{r\geq2}\sum_{k=0}^nc_{r,k}A_{r,k}^n(x)
\]
for some $c_{r,k}\geq0$.
Then $\I_n(p)\prec p$, and in particular, $p$ is both real-rooted and alternatingly increasing.
\end{theorem}

\begin{proof}
Notice first that 
\begin{equation*}
\begin{split}
(rx)^k(rx+1)^{n-k} 
	&= (rx)^k((r-1)x+(x+1))^{n-k},\\
	&=\sum_{j=0}^{n-k}{n-k\choose j}r^k(r-1)^jx^{k+j}(x+1)^{n-k-j}.\\
\end{split}
\end{equation*}
Thus, $\E((rx)^k(rx+1)^{n-k})\in\E_n$ for all $r\geq2$.  
So by Lemma~\ref{lem: first properties of A_d} and Proposition \ref{cones}, it suffices to find a polynomial $q\in\A_n$ such that $q\prec \E((rx)^k(rx+1)^{n-k})$ for all $r\geq2$.  
To this end, we recall that $A_{n,2}$ is well-known to be real-rooted and symmetric with respect to $n$ \cite{S94}.  
Thus, it holds that $E_{2,0}^n\in\A_n$.  
By \cite[Theorem 4.6]{B06}, if two polynomials $p,q$ with zeros $\alpha_1\leq \alpha_2 \leq \cdots \leq \alpha_n$ and $\beta_1\leq \beta_2 \leq \cdots \leq \beta_n$, respectively, that are all in the interval $[-1,0]$ satisfy $\alpha_k\leq \beta_k$ for all $k\in\{0,\ldots, n\}$ then $\E(p)\prec \E(q)$. 
Hence, it follows that $E_{2,0}^n\prec E_{r,k}^n$ for all $r\geq2$ and $0\leq k\leq n$. 
So the result follows.
\end{proof}

As an immediate corollary to Theorem~\ref{thm: colored eulerian polynomials}, we recover that the colored Eulerian polynomial $A_{n,r}(x)$ is interlaced by its own reciprocal.
%---COROLLARY: Colored Eulerian Polynomials---
\begin{corollary}
\label{cor: colored eulerian polynomials}
For all positive integers $n$ and $r$, the colored Eulerian polynomial $A_{n,r}$ satisfies
\[
\I_n(A_{n,r})\prec A_{n,r}.
\]
In particular, $A_{n,r}$ is both real-rooted and alternatingly increasing. 
\end{corollary}

\begin{proof}
Notice first that the result is well-known in the case when $r=1$, since $A_n$ is a real-rooted and symmetric polynomial.  
For $r>1$, the result follows immediately from Theorem~\ref{thm: colored eulerian polynomials}.
\end{proof}
We would now like to prove the analogous result to Corollary~\ref{cor: colored eulerian polynomials} for $d_{n,r}$.
To do this, we will make use of the following polynomial operator.

%---SUBSECTION: A Deranged Map----
\subsection{A deranged map}
\label{subsec: a deranged map}
Consider the linear map
\[
\D : \R[x]\to \R[x]
\quad 
\mbox{defined by }
\quad
\D(x^k)=d_k \mbox{ for all } k \geq 0. 
\]
Note that 
$$
\D(x+1)^n = \sum_{k=0}^n \binom n k d_k= A_n.
$$
We will now prove some properties of this transformation.
To do so, we require a few tools from the theory of \emph{(real) stable polynomials}; a multivariate generalization of real-rooted polynomials which was developed in \cite{BB09a,BB09b} and surveyed in \cite{W11}.  

Given $z\in\C$, let $\Re(z)$ and $\Im(z)$ respectively denote the real and imaginary parts of $z$, and let $\HH :=\{x\in\C : \Im(z)>0\}$.
We also denote a point $(z_1,\ldots,z_m)\in\C^m$ simply by $\bf{z}$.  
A polynomial $p\in\C[x_1,\ldots,x_m]$ is called \emph{stable} if $p$ is identically zero or $p({\bf z})\neq 0$ whenever ${\bf z}\in\HH^m$.  
For $p\in\C[x_1,\ldots,x_m]$ and $i\in[m]$ we let $\deg_i(p)$ denote the degree of $x_i$ in $p$.  
We will make use of some of the following elementary properties of stable polynomials.
%---LEMMA: Elementary Properties-----
\begin{lemma}
\cite[Lemma 2.4]{W11}
\label{lem: elementary properties}
The following operations preserve stability of polynomials:
\begin{enumerate}
	\item \emph{Permutation:} For $\pi\in\mathfrak{S}_m$, $p\mapsto p(x_{\pi(1)},\ldots,x_{\pi(m)})$,
	\item \emph{Scaling:} For $c\in\C$ and $\emph{c}\in\R^m$, $p\mapsto cp(c_1x_1,\ldots,c_mx_m)$,
	\item \emph{Diagonalization:} For $i,j\in[m]$, $p\mapsto p\Big|_{x_i = x_j}$,
	\item \emph{Specialization:} For $z\in\overline{\HH}$, the closure of $\HH$, $p\mapsto p(z,x_2,\ldots,x_m)$,
	\item \emph{Inversion:} $p\mapsto x_1^{\deg_1(f)}p(-x_1^{-1},x_2,\ldots,x_m)$, and 
	\item \emph{Differentiation:} $p\mapsto \dfrac{\partial}{\partial x_1}p$.
\end{enumerate}
\end{lemma}
By definition of $p\prec q$ for univariate polynomials, we have that $p\prec q$ if and only if the polynomial $p+iq\in\C[x]$ is a stable polynomial (see Theorem~\ref{thm: Hermite-Biehler}).  
Generalizing this, for $p,q\in\R[x_1,\ldots,x_m]$, we write $p\prec q$ if and only if $q+ip$ is stable.  
We will then use the following well-known lemma.
\begin{lemma}\label{dint}
If $p \in \R[x_1,\ldots,x_n]$ is stable and $\alpha_1,\ldots, \alpha_n$ are nonnegative numbers, then 
$$
\sum_{j=1}^n \alpha_j \dfrac{\partial p}{\partial x_j} \prec p.
$$
\end{lemma}
\begin{proof}
We want to prove that if $p\in \R[x_1,\ldots,x_n]$ is stable, then so is 
$$
p + i \sum_{j=1}^n \alpha_j \dfrac{\partial p}{\partial x_j}.
$$
However the differential operator $T=1 +i\sum_{j=1}^n \alpha_j {\partial}/{\partial x_j}$ is seen to preserve stability by, for example, an application of \cite[Theorem 1.2]{BB}.
\end{proof}

%---LEMMA: Deranged Map----
\begin{lemma}
\label{lem: deranged map}
Let $p\in\R[x]$ be a polynomial whose zeros all lie in the interval $[-1,0]$.  
Then $\D(p)$ is real-rooted.  
Moreover, $\D(p) \prec \D((x-\alpha)p)$ for all $\alpha \in [-1,0]$.
\end{lemma}

\begin{proof}
If $\theta=\{\theta_i\}_{i=1}^n$ is a multiset of $n$ numbers in the interval $[0,1]$, let $e_k(\theta)$ be the $k^{th}$ elementary symmetric function in $\theta$. 
We need to prove that the polynomial 
\[
G_\theta(x) := \sum_{k=0}^n e_{n-k}(\theta)d_k(x)
\]
is real--rooted. Note that $x^nd_n(1/x) =d_n(x)$ for all $n$. 
Hence, $G_\theta(x)$ is real-rooted if and only if the polynomial 
\[
H_\theta(x) :=\sum_{k=0}^n e_{n-k}(\theta)x^{n-k}d_k(x)
\]
is real-rooted. 
Let 
\[
F_\theta(x,y) := 
\sum_{\pi \in \mathfrak{S}_n}x^{\exc(\pi)}y^{\axc(\pi)} \prod_{i}x\theta_i,
\]
where the product is over all fixed points $i$ of $\pi$, and where $\axc(\pi)$ is the number of \emph{anti-excedances} in $\pi$; i.e., the number of indices $i\in[n]$ for which $\pi_i<i$. 
Note that $F_\theta(x,y)$ is homogeneous and $F_\theta(x,1)=H_\theta(x)$. We claim that 
\[
F_\theta= \theta_1 x F_{\theta \setminus \theta_1}+ xy\left(\frac \partial {\partial x} + \frac \partial {\partial y}\right)F_{\theta\setminus \theta_1}+ xy\sum_{i=2}^n (1-\theta_i)F_{\theta \setminus \{\theta_1,\theta_i\}}.
\]
The proof of the claim follows by examining the effect of inserting $1$ in a permutation of $\{2,\ldots,n\}$.  
Here, we can think of the permutation of $\{2,\ldots,n\}$ as given in cycle notation, and inserting $1$ as placing $1$ in between two elements of $\{2,\ldots,n\}$ in one of the cycles of this permutation (see also \cite[Proof of Theorem 5.4]{BLV16}).  
The first term on the right-hand-side of the above equation corresponds to making $1$ a fixed point. 
The second term corresponds to inserting $1$ into a \emph{weak excedance} (i.e., an index $i\in[n]$ for which $\pi_i\geq i$) or an anti-excedance. 
However, in doing so, we have counted wrong when inserting $1$ into a fixed point. 
Indeed we got an extra $\theta_i$. 
Hence, for any $\pi \in \mathfrak{S}_n$ where $\{1,i\}$, $2\leq i \leq n$, is a two-cycle we need to compensate with 
\[
(1-\theta_i) x^{\exc(\pi)}y^{\axc(\pi)} \prod_{j}x\theta_j.
\]
This corresponds to the last term. 

We now prove by induction over $n$ that $H_{\theta\setminus \theta_1} \prec H_{\theta}$. 
Note that by homogeneity, $H_\theta$ is real-rooted if and only if $F_\theta$  is stable (see \cite[Theorem 4.5]{BBL}). 
Now 
\[
H_\theta = \theta_1 x H_{\theta \setminus \theta_1}+ x\left(\frac \partial {\partial x} + \frac \partial {\partial y}\right)F_{\theta\setminus \theta_1}(x,1)+ x\sum_{i=2}^n (1-\theta_i)H_{\theta \setminus \{\theta_1,\theta_i\}}.
\]
Since $F_{\theta\setminus \theta_1}$ is stable by induction, and since we have  
\[
\left(\frac \partial {\partial x} + \frac \partial {\partial y}\right)F_{\theta\setminus \theta_1} \prec F_{\theta\setminus \theta_1}
\]
by Lemma \ref{dint}, then $H_{\theta \setminus \theta_1} \prec x \left(\frac \partial {\partial x} + \frac \partial {\partial y}\right)F_{\theta\setminus \theta_1}(x,1)$. 
Also, 
$H_{\theta \setminus \theta_1} \prec \theta_1 x H_{\theta \setminus \theta_1}$ and $H_{\theta \setminus \theta_1} \prec xH_{\theta \setminus \{\theta_1,\theta_i\}}$, by induction. 
The theorem now follows from Proposition \ref{cones}, since the statement $\D(p) \prec \D((x-\alpha)p)$ for all $\alpha \in [-1,0]$ corresponds to $H_{\theta\setminus \theta_1} \prec H_{\theta}$. 
\end{proof}

Given two polynomials $p$ and $q$ of degree $n$ with zeros $\alpha_1\leq\alpha_2\leq\cdots\leq\alpha_n$ and $\beta_1\leq\beta_2\leq \cdots\leq \beta_n$ we write $p\leq q$ whenever $\alpha_k\leq \beta_k$ for all $k\in[n]$. 
Using Lemma~\ref{lem: deranged map}, we can prove the following theorem in analogy to \cite[Theorem 4.6]{B06}, which applies to the subdivision operator.

%---THEOREM: Deranged Map---
\begin{theorem}
\label{thm: deranged map}
Suppose that $p,q\in\R[x]$ are polynomials both of which have all zeros in the interval $[-1,0]$, and suppose that $p\leq q$.  
Then $\D(p)\prec\D(q)$.  
\end{theorem}

\begin{proof}
We prove the statement by induction on the degree of $p$ and $q$.  
Suppose first that $p$ and $q$ differ by exactly one zero and satisfy $p\leq q$.
Then $p = (x-\alpha)h$ and $q = (x-\beta)h$ for some $\alpha,\beta\in[-1,0]$ with $\alpha\leq \beta$ and $h$ a polynomial with all of its zeros in the interval $[-1,0]$.  
Thus, 
\[
\D(p) = \D(q)+(\beta-\alpha)\D(h).
\]
Since $\D(q)\prec\D(q)$ and $\D(h)\prec\D(q)$, by Lemma~\ref{lem: deranged map}, it follows from Proposition \ref{cones} that $\D(p)\prec \D(q)$.

Now, suppose that $p$ and $q$ have degree $n$, have all zeros in the interval $[-1,0]$, and that $p\leq q$.  
Then there exists a sequence of polynomials $(x+1)^n=h_0 \leq h_1 \leq \cdots \leq h_m=x^n$ containing 
$p$ and $q$ such that $h_{i-1}$ and $h_{i}$ differ only by one zero for all $i \in [n]$. 
By the preceding argument, we have that 
\[
A_n = \D(h_0)\prec\D(h_1)\prec\cdots\prec\D(h_{m-1})\prec\D(h_m) = d_n.
\]
To complete the proof it suffices to prove that $A_n\prec d_n$ since the result then follows by applying \cite[Lemma 2.3]{B06}.
To see this, notice first that the preceding argument implies that we have a chain of interlacing relations:
\[
A_n = p_0\prec p_1\prec\cdots\prec p_{m-1}\prec p_m = d_n.
\]
Let $q_i:=\I_n(p_i)(x)$ for all $0\leq i\leq m$, and notice that $q_{i}\prec q_{i-1}$.  
Moreover, since $\I_n(A_n) = xA_n$ and $\I_n(d_n) = d_n$ then
\[
A_n = p_0\prec\cdots\prec p_{m-1}\prec p_m = d_n = q_m\prec q_{m-1}\prec\cdots\prec q_0 = xA_n.
\]
Since $A_n\prec xA_n$, we have $A_n\prec d_n$ by \cite[Lemma 2.3]{B06}, which completes the proof.
\end{proof}

%---COROLLARY: Deranged Map----
\begin{corollary}
\label{cor: deranged map}
Suppose $p =\sum_{k=0}^nc_kx^k(x+1)^{n-k}$, where $c_k\geq0$ for all $0\leq k \leq n$. 
Then $\D(p)$ is real-rooted, $A_n\prec \D(p)\prec d_n$, and
\[
\D(p)\prec \I_n(\D(p)).
\]
In particular, $A_n\prec d_n$.
\end{corollary}

\begin{proof}
Since $\D(p) = \sum_{k=0}^nc_k\D(x^k(x+1)^{n-k})$, Theorem~\ref{thm: deranged map} together with Proposition \ref{cones} imply 
\[
A_n\prec \D(p)\prec d_n.
\]
Since $\I_n(A_n) = xA_n$, $\I_n(d_n) = d_n$, and $A_n \prec xA_n$, we have
\[
A_n \prec \D(p)\prec d_n \prec \I_n(\D(p))\prec xA_n, 
\]
and hence 
\[
(x-\alpha)A_n \prec \D(p)\prec d_n \prec \I_n(\D(p))\prec xA_n, 
\]
for $\alpha<0$ sufficiently small. 
Since $(x-\alpha)A_n \prec xA_n$, the proof now follows from \cite[Lemma 2.3]{B06}. 
\end{proof}

%---SUBSECTION: Colored Derangement Polynomials-------
\subsection{The colored derangement polynomials and some local $h$-polynomials}
\label{subsec: the colored derangement polynomials}
Using Corollary~\ref{cor: deranged map}, we can now show that all colored derangement polynomials $d_{n,r}$ satisfy the interlacing properties of Theorem~\ref{thm: interlacing TFAE}, thereby strengthening and generalizing the results of \cite{CTZ09,C09}.  
We then apply this result to answer some recent conjectures on families of local $h$-polynomials.
%---THEOREM: Interlacing Properties for Derangement Polynomials----
\begin{theorem}
\label{thm: interlacing properties for derangement polynomials}
For positive integers $n$ and $r$ the derangement polynomial $d_{n,r}$ satisfies the interlacing property
\[
\I_n(d_{n,r})\prec d_{n,r}.
\]
In particular, $d_{n,r}$ is both real-rooted and alternatingly increasing.
\end{theorem}

\begin{proof}
For $r=1$, the statement follows from the fact that $\I_n(d_n) = d_n$, and that $d_n$ has only real-zeros.  
Let $r>1$.  
Then  
\[
d_{n,r} = \sum_{k=0}^n{n\choose n-k}(r-1)^{n-k}r^kx^{n-k}d_k, 
\]
and since $\I_n(d_n)= d_n$ 
\[
\I_n(d_{n,r}) = \sum_{k=0}^n{n\choose n-k}(r-1)^{n-k}r^kd_k = \D((rx+(r-1))^n).
\]
By expanding $(rx+(r-1))^n$ as
\[
(rx+(r-1))^n = (x+ (r-1)(x+1))^n = \sum_{k=0}^n{n\choose k}(r-1)^{n-k}x^k(x+1)^{n-k},
\]
and applying Corollary~\ref{cor: deranged map} we see that 
\[
\I_n(d_{n,r})= \D(rx+(r-1))^n \prec \I_n \D(rx+(r-1))^n=d_{n,r},
\]
from which the theorem follows. 
(Recall from Remark~\ref{rmk: alternatingly increasing connection} that $\I_n(d_{n,r})\prec d_{n,r}$ implies that $d_{n,r}$ is both real-rooted and alternatingly increasing.
\end{proof}

An immediate consequence of Theorem~\ref{thm: interlacing properties for derangement polynomials} is that the $\I_n$-decomposition $(a,b)$ of $d_{n,r}$ is always real-rooted.  
The polynomial $a$ was shown to be the local $h$-polynomial of the $r^{th}$ edgewise subdivision of the barycentric subdivision of a simplex in \cite[Theorem 1.2]{A14}.
(We refer the reader to \cite{A16} for all necessary definitions pertaining to local $h$-polynomials and subdivisions.)
This polynomial, as well as the polynomial $a$ in $(a,b)$, was conjectured by Athanasiadis to be real-rooted \cite[Question 4.11]{A16}, \cite[Conjecture 2.30]{A17}, \cite[Conjecture 3.7.10]{S13}.  
An affirmative answer to these conjectures is now an immediate corollary of Theorem~\ref{thm: interlacing properties for derangement polynomials}.  

%---COROLLARY: Real-rooted I-decomposition for derangement polynomials----
\begin{corollary}
\label{cor: real-rooted I-decomposition for derangement polynomials}
For positive integers $n$ and $r$, the $\I_n$-decomposition of $d_{n,r}$ is real-rooted.  
Consequently, the local $h$-polynomial of the $r^{th}$ edgewise subdivision of the barycentric subdivision of a simplex is real-rooted.
\end{corollary}

\begin{proof}
The real-rootedness of $a$ and $b$ follows from combining Theorem~\ref{thm: interlacing properties for derangement polynomials} with Theorem~\ref{thm: interlacing TFAE}.  
The consequences for local $h$-polynomials then follows from \cite[Theorem 1.2]{A14}.
\end{proof}

%---REMARK: Second Proof----
\begin{remark}
\label{rmk: second proof}
We note that a second proof of Corollary~\ref{cor: real-rooted I-decomposition for derangement polynomials} was (subsequently) given by the second author and N.~Gustafsson in \cite{GS18}.  
\end{remark}

%---SECTION: h*-polynomials of Lattice Zonotopes------------
\section{$h^\ast$-polynomials for Combinatorially-Positive Valuations of Lattice Zonotopes}
\label{sec: h-polynomials of lattice zonotopes}

Let $P\subset \R^n$ be a $d$-dimensional lattice polytope; i.e., a polytope whose vertices all lie in the lattice $\Z^n$.  
The \emph{Ehrhart function} of $P$ is the lattice-point enumerator $i(P;m) :=\left|mP\cap\Z^n\right|$, where $m\in \Z_{\ge0}$ and $mP:=\{mp : p\in P\}$ is the $m^{th}$ \emph{dilate} of $P$.  
The \emph{Ehrhart series} of $P$ is the generating function
\[
\sum_{m\ge0}i(P;m)x^m = \frac{h_0^\ast+h_1^\ast x+\cdots +h_d^\ast x^d}{(1-x)^{d+1}}.
\]
The polynomial $h^\ast(P;x) :=h_0^\ast+h_1^\ast x+\cdots +h_d^\ast x^d$ is called the \emph{(Ehrhart)} $h^\ast$-\emph{polynomial} of $P$, and it always has only nonnegative integer coefficients \cite{S80}.  
The function $i(P;m)$ is additionally a polynomial of degree $d$ \cite{E62}, called the \emph{Ehrhart polynomial} of $P$. %and is thus an $h$-polynomial in the sense of equation~\eqref{eqn:1}.
The distributional properties of $h^\ast(P;x)$, such as unimodality and real-rootedness, are highly-investigated \cite{B16}, and the following observation has recently raised new questions about the alternatingly increasing property for $h^\ast$-polynomials.
%---THEOREM: Betke McMullen----
\begin{theorem}
\label{thm: betke mcmullen}
\cite[Theorem 10.5]{BR07}
Let $P$ be a $d$-dimensional lattice polytope containing a lattice point in its relative interior.  
Then both polynomials in the $\I_d$--decomposition of $h^\ast(P;x)$ have only nonnegative integer coefficients.  
\end{theorem}

A lattice polytope $P$ is said to be \emph{integrally closed} or \emph{IDP} if for all $m\in\Z_{>0}$ and $x\in mP\cap\Z^n$ there exist $x^{(1)},\ldots,x^{(m)}\in P\cap\Z^n$ such that $x = x^{(1)}+\cdots+x^{(m)}$.  
The following conjecture is stated in \cite{BJM16} who attribute it to \cite{SV13}.
%---CONJECTURE: IDP --------------
\begin{conjecture}
\label{conj: idp}
\cite[Conjecture 1]{BJM16}
Let $P$ be an IDP lattice polytope with a lattice point in its relative interior.  
Then $h^\ast(P;x)$ is alternatingly increasing.
\end{conjecture}

In \cite{SV13}, the authors verified Conjecture~\ref{conj: idp} for lattice parallelpipeds, and in \cite{BJM16}, the authors extended these results to include centrally symmetric lattice zonotopes as well as coloop-free lattice zonotopes.  
In Subsection~\ref{subsec: alternatingly increasing property for lattice zonotopes} we generalize and strengthen these results by proving that the $h^\ast$-polynomial with respect to any \emph{combinatorially-positive valuation} \cite{JS15} of a lattice zonotope containing a lattice point in its relative interior has a real-rooted $\I$-decomposition.  
In Subsection~\ref{subsec: interlacing properties for centrally symmetric lattice zonotopes} we further strengthen these results by proving that any such $h^\ast$-polynomial for a centrally symmetric lattice zonotope satisfies the interlacing conditions of Theorem~\ref{thm: interlacing TFAE}.  
First, in Subsection~\ref{subsec: lattice-invariant valuations}, we recall the theory of combinatorially-positive valuations and their resulting generalization of Ehrhart $h^\ast$-polynomials, which will be considered in the coming subsections.   

%---SUBSECTION: Lattice Invariant Valuations-----
\subsection{Lattice-invariant valuations}
\label{subsec: lattice-invariant valuations}
Let $\PP(\Z^n)$ denote the collection of lattice polytopes in $\R^n$.  
A \emph{(real) valuation} on lattice polytopes is a map $\varphi:\PP(\Z^n)\longrightarrow \R$ that satisfies
\[
\varphi(P\cup Q) = \varphi(P)+\varphi(Q)-\varphi(P\cap Q),
\]
whenever $P\cup Q$ and $P\cap Q$ are also lattice polytopes.  
A valuation $\varphi$ is further called \emph{translation-invariant} if $\varphi(P+z) = \varphi(P)$, for all $z\in\Z^n$ and $P\in\PP(\Z^n)$. 
McMullen proved that for all $m\in\Z_{\geq0}$ and $P\in\PP(\Z^n)$, the value $\varphi(mP)$ is given by a polynomial $i^\varphi(P;x)$ of degree at most $d$.  
Equivalently, there exist $h_0^\varphi(P),\ldots,h_d^\varphi(P)\in \R$ such that
\[
\sum_{m\ge0}i^\varphi(P;m)x^m = \frac{h_0^\varphi(P)+h_1^\varphi(P) x+\cdots +h_d^\varphi(P) x^d}{(1-x)^{d+1}}.
\]
In the special case when $\varphi(P)  = \left|P\cap\Z^n\right|$ the polynomial $i^\varphi(P;x)$ is the Ehrhart polynomial of $P$, and the polynomial $h^\varphi(P;x) = h_0^\varphi(P)+h_1^\varphi(P) x+\cdots +h_d^\varphi(P) x^d$ is the $h^\ast$-polynomial.  
Consequently, for a given valuation $\varphi$, the polynomial $h^\varphi(P;x)$ is called the $h^\ast$-\emph{polynomial of} $P$ \emph{with respect to $\varphi$}.  
In the case when $\varphi$ is a translation-invariant valuation such that $h^\varphi(P;x)$ has nonnegative coefficients for all $P$, we say that it is \emph{combinatorially positive} \cite{JS15}.  
Via a change of basis, the Ehrhart polynomial of a lattice polytope $P$ can also be expressed as
\[
i(P;x) = \sum_{j = 0}^df^\ast_j(P){x-1\choose j},
\]
where $f^\ast_j(P)\in\Z$ \cite[Theorem 3]{B12}.  
A valuation $\varphi:\PP(\Z^n)\longrightarrow \R$ is further called \emph{lattice-invariant} if $\varphi(T(P)) = \varphi(P)$ for all $P\in\PP(\Z^n)$ and every affine map satisfying $T(\Z^n) = Z^n$.  
We will make use of the following characterization of combinatorially positivity for lattice-invariant valuations.
%---THEOREM: Valuation Characterization------
\begin{theorem}
\label{thm: valuation characterization}
\cite[Theorem 7.1]{JS15}
Let $\varphi: \PP(\Z^n)\longrightarrow \R$ be a lattice-invariant valuation.  
Then $\varphi$ is combinatorially-positive if and only if there exist $\alpha_0,\ldots,\alpha_d\in \R_{\geq0}$ such that
\[
\varphi = \alpha_0f^\ast_0+\alpha_1f^\ast_1+\cdots+\alpha_df^\ast_d.
\]
\end{theorem}

Suppose that $\varphi$ is a combinatorially-positive valuation.  
Then by Theorem~\ref{thm: valuation characterization} there exist $\alpha_0,\ldots,\alpha_d\in\R_{\geq0}$ such that for all lattice polytopes $P\in\PP(\Z^n)$
\[
\varphi(P) = \sum_{k=0}^d\alpha_kf^\ast_k(P).
\]
Define $i^k(P;x)$ to be the polynomial satisfying $i^k(P;m) = f^\ast_k(mP)$ for all $m\in\Z_{\geq0}$.  
Then
\[
i(P;xm) = i(mP;x) = \sum_{k\geq0}f^\ast_k(mP){x-1\choose k}.
\]
Thus, 
\begin{equation}
\label{nuk}
f^\ast_k(mP) = \Delta_x^k(i(P;xm))\Big|_{x=1}, 
\end{equation}
where $\Delta_x$ is the forward-difference operator $\Delta_x(f)(x)= f(x+1)-f(x)$, see \cite[Chapter 1.9]{S11}. 
Hence, there is a unique linear operator $T_k : \R[x] \longrightarrow \R[x]$ for which $T_k(i(P;x)) = i^k(P;x)$ for all polytopes $P$. 
The following theorem provides a description of the operator $T_k$ in terms of the Stirling numbers of the second kind, which we denote by $S(m,k)$, see \cite{S11}. 
%---THEOREM: Tk Operator----
\begin{theorem}
\label{thm: tk operator}
The operator $T_k$ is the unique operator satisfying 
\[
T_k(x^m) = k!S(m+1,k+1)x^m,
\]
for all $m \geq 0$. 
\end{theorem}

\begin{proof}
From equation~\eqref{nuk} we see that 
\[
T_k(f)(x)= \sum_{i=0}^k(-1)^{k-i} \binom k i f((i+1)x).
\]
Hence, 
\begin{align*}
T_k(x^m)/x^m &= \sum_{i=0}^k(-1)^{k-i} \binom k i (i+1)^m \\
&= \sum_{j=0}^m \binom m j \sum_{i=0}^k (-1)^{k-i} \binom k i i^j \ \ \ \ \mbox{(see \cite[Chapter 1.9]{S11})}\\
&= \sum_{j=0}^m \binom m j k! S(j,k) = k!S(m+1,k+1), 
\end{align*}
\end{proof}
Notice that we now have the following relationship via the operator $T_k$ between the Ehrhart polynomial $i(P;x)$ and that of any combinatorially positive valuation:
\begin{equation}
\label{eqn: combinatorially-positive}
i^\varphi(P;x) = \sum_{k=0}^d\alpha_kT_k(i(P;x)).
\end{equation}
Thus, if the Ehrhart polynomial $i(P;x)$ has a particularly nice form, such as that of Theorem~\ref{thm: real-rooted I-decomposition}, we can hope to translate the nice form of $i(P;x)$ to a similar one for $i^\varphi(P;x)$.
The next lemma does precisely this in regards to the inequalities~\eqref{eqn: hibi} from Theorem~\ref{thm: real-rooted I-decomposition}.  
In the following, let $\phi_k(f)= T_k(f)/k!$. 
From Theorem~\ref{thm: tk operator} and the standard recursion for $S(m,k)$, namely, 
\[
S(m+1,k) = kS(m,k) + S(m,k-1)
\]
for $k>0$, $S(0,0) = 1$, and $S(n,0) = S(0,n) = 0$ for $n>0$,
we get that
\begin{equation}
\label{eqn: recursion}
\phi_k(xf)= (k+1)x\phi_k(f) + x\phi_{k-1}(f),
\end{equation}
for all $k \geq 1$. 
This recursion will be key in our analysis of the operator $T_k$ in the coming applications.
%---LEMMA: Cones----
\begin{lemma}
\label{lem: cones}
Let $d,k$ and $\gamma$ be nonnegative integers and let $\B_{d,\gamma}$ denote the nonnegative span of all polynomials of the form 
\begin{equation}
\label{gene}
x^i(x+1)^{d-i}+x^j(x+1)^{d-j},
\end{equation}
where $i+j \geq \gamma$.  
Then 
\[
\phi_k(\B_{d,\gamma}) \subseteq \B_{d,\gamma}.
\]
\end{lemma}

\begin{proof}
The proof is by induction on $d$. 
Suppose the result is true for $d-1$. 
We need to prove that the image of each generator \eqref{gene}, where $i\leq j$, is in $\B_{d,\gamma}$. 

Suppose first that $i=0$. 
Then $j \geq \gamma$, and we may write 
\begin{equation}
\label{eqn: bijective}
\begin{split}
\phi_k((x+1)^d) &= \sum_{\ell=0}^d a_\ell x^\ell(x+1)^{d-\ell}, \ \ \ \mbox{ and }\\
\phi_k(x^j(x+1)^{d-j}) &= \sum_{\ell=j}^d b_\ell x^\ell(x+1)^{d-\ell},
\end{split}
\end{equation}
where $\sum_{\ell=0}^d a_\ell =\sum_{\ell=j}^d b_\ell = S(d+1,k+1)$, and all $a_\ell$ and $b_\ell$ are nonnegative integers. 
To see this, consider the following for $0\leq j\leq d$:
\begin{align*}
T_k(x^j(x+1)^{d-j})
&= \sum_{i=0}^k(-1)^{k-i}{k\choose i}(i+1)^jx^j((i+1)x+1)^{d-j},	\\
&= \sum_{i=0}^k(-1)^{k-i}{k\choose i}(i+1)^jx^j((ix)+(x+1))^{d-j},	\\
&= \sum_{i=0}^k(-1)^{k-i}{k\choose i}(i+1)^j\sum_{\ell = 0}^{d-j}{d-j\choose \ell}i^\ell x^{j+\ell}(x+1)^{d-(j+\ell)},	\\
&= \sum_{\ell=0}^{d-j}{d-j\choose \ell}\left(\sum_{i=0}^k(-1)^{k-i}{k\choose i}\sum_{r=0}^j{j\choose r}i^{r+\ell}\right) x^{j+\ell}(x+1)^{d-(j+\ell)},	\\
&= \sum_{\ell=0}^{d-j}{d-j\choose \ell}\sum_{r=0}^j{j\choose r}\left(\sum_{i=0}^k(-1)^{k-i}{k\choose i}i^{r+\ell}\right) x^{j+\ell}(x+1)^{d-(j+\ell)},	\\
&= k!\sum_{\ell=0}^{d-j}{d-j\choose \ell}\left(\sum_{r=0}^j{j\choose r}S(r+\ell,k)\right) x^{j+\ell}(x+1)^{d-(j+\ell)}.
\end{align*}
It follows from the last line that the coefficients $a_\ell$ and $b_\ell$ in equations~\eqref{eqn: bijective} are nonnegative integers.  
To see that the coefficients also have the desired sum simply notice that
\[
\sum_{\ell=0}^{d-j}\sum_{r=0}^j{d-j\choose\ell}{j\choose r}S(r+\ell,k) = \sum_{m=0}^d{d\choose m}S(m,k) = S(d+1,k+1).
\]
It follows that we may match each term in the first sum of equation~\eqref{eqn: bijective} to a term in the second sum bijectively and write 
$
\phi_k((x+1)^d+x^j(x+1)^{d-j})
$
as a sum of terms of the form $x^\ell(x+1)^{d-\ell}+ x^m(x+1)^{d-m}$ where $\ell \geq 0$ and $m \geq j \geq \gamma$. 
This proves the case when $i=0$. 

If $i >0$, let $g=x^i(x+1)^{d-i}+x^j(x+1)^{d-j} \in \B_{d,\gamma}$. 
Then $g=xh$, where $h \in \B_{d-1,\gamma-2}$.  
Hence, applying the recursion~\eqref{eqn: recursion} we get that
\[
\phi(g)=\phi_k(xh)= (k+1)x\phi_k(h) + x\phi_{k-1}(h).
\]
By induction $\phi_k(h), \phi_{k-1}(h) \in  \B_{d-1,\gamma-2}$, and thus 
$x\phi_k(h), x\phi_{k-1}(h) \in  \B_{d,\gamma}$.
\end{proof}

It follows from Lemma~\ref{lem: cones} that if $i(P;x)$ satisfies the assumptions of Theorem~\ref{thm: real-rooted I-decomposition}, including inequalities \eqref{eqn: hibi}, then so does $i^\varphi(P;x)$ for any combinatorially-positive valuation.  
This is the key to proving $h^\varphi(P;x)$ is alternatingly increasing for all lattice zonotopes with interior lattice points.  

%---SUBSECTION: The Alternatingly Increasing Property for Lattice Zonotopes-----
\subsection{The alternatingly increasing property for lattice zonotopes}
\label{subsec: alternatingly increasing property for lattice zonotopes}
Given a point $u\in\R^n$, let ${\bf u}$ denote the line segment in $\R^n$ given by the convex hull of $u$ and the origin.  
A \emph{lattice zonotope} is a lattice polytope that is the Minkowski sum of line segments $\bf{ u^{(1)}},\ldots,\bf{ u^{(m)}}$, denoted $\Zone(u^{(1)},\ldots,u^{(m)})$.  
If the points $u^{(1)},\ldots,u^{(m)}$ are understood from context, or irrelevant, we may write simply $\Zone$.  
In \cite{D15}, the author showed that Ehrhart polynomial of lattice zonotope $\Zone(u^{(1)},\ldots,u^{(m)})$ is closely related to the $h^\ast$-polynomial of a second polytope associated to the vectors $u^{(1)},\ldots,u^{(m)}$. 
Let $U$ denote the $n\times m$ matrix whose columns are $u^{(1)},\ldots,u^{(m)}$.  
The \emph{Lawrence matrix} of $u^{(1)},\ldots,u^{(m)}$ is the block matrix
\[
\begin{pmatrix}
U	&	\emph{0}_{n\times m}	\\
I_m	&	I_m				\\
\end{pmatrix},
\]
where $\emph{0}_{n\times m}$ denotes the $n\times m$ matrix of all zeros and $I_m$ denotes the $m\times m$ identity matrix.
The \emph{Lawrence polytope} associated to $u^{(1)},\ldots,u^{(m)}$ is the convex hull of all columns of the Lawrence matrix of $U$.  
It is denoted by $\Lambda(u^{(1)},\ldots,u^{(m)})$, or simply $\Lambda$ if $u^{(1)},\ldots,u^{(m)}$ are understood from context or irrelevant.

%---THEOREM: Aaron's Theorem---
\begin{theorem}
\cite[Theorem 3.7]{D15}
\label{thm: aaron's theorem}
Let $u^{(1)},\ldots,u^{(m)}\subset\Z^d$ be a finite set of vectors defining a $d$-dimensional lattice zonotope $\Zone = \Zone(u^{(1)},\ldots,u^{(m)})$.  
Then the Ehrhart polynomial $\Zone$ is related to the $h^\ast$-polynomial of the associated Lawrence polytope $\Lambda$ via
\[
i\left(\Zone;x\right) = \sum_{i=0}^dh^\ast_i\left(\Lambda\right)x^i(x+1)^{d-i}.
\]
\end{theorem}

We note that in the original statement of Theorem~\ref{thm: aaron's theorem} in \cite{D15}, the author further includes the hypothesis that $u^{(1)},\ldots,u^{(m)}$ span $\Z^d$ so that they can apply \cite[Theorem 1.6]{S09b} in their proof.  
However, the combinatorial proof of \cite[Theorem 1.6]{S09b} does not actually require this hypothesis. It only requires the combinatorial interpretation of a regular triangulation of a Lawrence polytope $\Lambda(u^{(1)},\ldots,u^{(m)})$ in terms of an associated matroid.  
The additional hypothesis of spanning is necessary in \cite{S09b} to ensure that the coefficients of the $h^\ast$-polynomial of the Lawrence polytope admit an interpretation in terms of cohomology groups of hypertoric varieties.  
Since we do not need this connection, we are free to use the more general form of the result stated in Theorem~\ref{thm: aaron's theorem}. 
Thus, we can derive the following theorem.

%---THEOREM: Zonotopes----
\begin{theorem}
\label{thm: zonotopes}
Let $\Zone$ be a $d$-dimensional lattice zonotope that has a lattice-point in its relative interior and $\varphi$ a combinatorially-positive valuation.
Then the $\I$--decomposition of $h^\varphi\left(\Zone;x\right)$ is real-rooted.
\end{theorem}

\begin{proof}
Let us first consider the combinatorially-positive valuation corresponding to the Ehrhart polynomial $i(P;x)$.  
Assume $\Zone  = \Zone(u^{(1)},\ldots,u^{(m)})$ for some vectors $u^{(1)},\ldots,u^{(m)}\in\Z^d$.  
By Theorem~\ref{thm: real-rooted I-decomposition} and Theorem~\ref{thm: aaron's theorem}, it suffices to prove that 
\[
h_0^\ast(\Lambda)+\cdots+h_i^\ast(\Lambda) \leq h_d^\ast(\Lambda)+\cdots+h_{d-i}^\ast(\Lambda),
\]
for all $0\leq i\leq d$.  
We note that, in this case, the dimension of $\Lambda$ is well-known to be $d+m-1$, and $h^\ast(\Lambda;x)$ has degree at most $d+m-1$.  
Thus, by \cite[Theorem 2.19]{S09}, it suffices to prove that the degree of $h^\ast(\Lambda;x)$ is actually $d$.  
First, we note that by the formula given for the coefficients $h^\ast_j(\Lambda)$ the proof of Theorem~\ref{thm: aaron's theorem} in \cite{D15}, it is clear that $h^\ast_j(\Lambda) = 0$ whenever $j>d$.  
Thus, it remains to check that $h^\ast_d(\Lambda) \neq 0$.  
This can be seen as follows:

By Ehrhart-MacDonald reciprocity \cite[Theorem 4.1]{BR07} we know that
\[
i(\Zone;-x) = (-1)^{d}i(\Zone^\circ;x),
\]
where $\Zone^\circ$ denotes the interior of the lattice zonotope $\Zone$.  
Thus, it follows that
\[
i(\Zone^\circ;1) = (-1)^{d}i(\Zone;-1) = \sum_{k=0}^dh_k^\ast(\Lambda)(-1)^{d+i}(-1+1)^{d-i} = h_d^\ast(\Lambda).
\]
Therefore, $h_d^\ast(\Lambda)\neq0$ since $\Zone$ contains a lattice point in its relative interior.  
Thus, we conclude that $h^\ast(\Lambda;x)$ has degree $d$.    
By Lemma~\ref{lem: cones} it follows that if $\varphi$ is a combinatorially-positive valuation then 
\[
i^\varphi(Z;x) = \sum_{k=0}^dc_k^\varphi x^k(x+1)^{d-k},
\]
where the coefficients $c_0^\varphi,\ldots,c_d^\varphi$ are nonnegative integers satisfying the inequalities (2) of Theorem~\ref{thm: real-rooted I-decomposition}.  
An application of Theorem~\ref{thm: real-rooted I-decomposition} thereby completes the proof.
\end{proof}

As a special case of Theorem~\ref{thm: zonotopes}, we have the following in support of Conjecture~\ref{conj: idp}.
%---COROLLARY: On IDP Conjecture----
\begin{corollary}
\label{cor: on idp conjecture}
Let $\Zone$ be a lattice zonotope with a lattice point in its relative interior. 
Then $h^\ast\left(\Zone;x\right)$ is alternatingly increasing.
\end{corollary}

%---SUBSECTION: Interlacing Properties for Centrally Symmetric Lattice Zonotopes-----
\subsection{Interlacing properties for centrally symmetric lattice zonotopes}
\label{subsec: interlacing properties for centrally symmetric lattice zonotopes}
It is natural to ask if the $h^\ast$-polynomials of lattice zonotopes further possess the interlacing properties of Theorem~\ref{thm: interlacing TFAE}, similar to the simplices for numeral systems studied in \cite{S17}.  
It turns out that the answer is ``yes'' when the lattice zonotope is centrally symmetric.
In \cite{BJM16} the authors proved that the Ehrhart $h^\ast$-polynomial of a centrally symmetric zonotope is always alternatingly increasing.  
These results are generalized by Theorem~\ref{thm: zonotopes}.  
In this subsection we further strengthen and generalize the results of \cite{BJM16} on centrally symmetric zonotopes by proving that the $h^\ast$-polynomial of any centrally symmetric zonotope with respect to any combinatorially-positive valuation satisfies the interlacing conditions of Theorem~\ref{thm: interlacing TFAE}.  

A $d$-dimensional lattice zonotope $Z$ is called a \emph{centrally symmetric zonotope (or Type-B zonotope)} if it can be realized as the projection of the cube $[-1,1]^m$ for some $m\geq d$.  
The name ``Type-B zonotope'' arises since the Ehrhart $h^\ast$-polynomial of such a zonotope is computable in terms of descent statistics for Type-B permutations \cite{BJM16}.  
A \emph{signed permutation} on $[d]$ is a pair $(\pi,\varepsilon)$ such that $\pi\in\mathfrak{S}_{d}$ and $\varepsilon\in\{-1,1\}^d$.  
The collection of all such pairs is denoted $\B_d$.  
Given $(\pi,\varepsilon)\in \B_d$, for each letter $\pi_i$ in the permutation $\pi$ we assign the sign $\varepsilon_i$.  
Set $\pi_0:=0$ and $\varepsilon_0:=1$.
The \emph{descent set} of $(\pi,\varepsilon)\in\B_d$ is the collection
\[
\Des(\pi,\varepsilon) :=
\{i\in[d-1]\cup\{0\}  
\, : \,
\varepsilon_i\pi_i>\varepsilon_{i+1}\pi_{i+1}
\}.
\]
%An index $i\in[d-1]\cup\{0\}$ is called a \emph{descent} of $(\pi,\varepsilon)$ if and only if $i\in\Des(\pi,\varepsilon)$
An element of $\Des(\pi,\varepsilon)$ is called a \emph{descent} of $(\pi,\varepsilon)$, and the \emph{descent number} of $(\pi,\varepsilon)$ is $\des(\pi,\varepsilon):=\left|\Des(\pi,\varepsilon)\right|$. 
The \emph{Type-B Eulerian polynomial} is defined to be 
\[
B_0^d(x) :=\sum_{(\pi,\varepsilon)\in \B_d}x^{\des(\pi,\varepsilon)},
\]
and it is known to be a real-rooted polynomial \cite{B94}.  
For $1\leq k \leq d$ define the polynomials
\[
B_k^d(x) :=\sum_{(\pi,\varepsilon)\in B_d}\chi\left(\varepsilon_d = 1,\pi_d = d+1-k\right)x^{\des(\pi,\varepsilon)},
\]
where $\chi(\varphi)$ equals $1$ if $\varphi$ is a true statement and $0$ otherwise.  
Note in particular that $B_1^{d+1} = B_0^d$.
For $0\leq k\leq d$, we define the \emph{half-open $(\pm1)$-cube} to be 
\[
[-1,1]_k^d := [-1,1]^d\backslash\{x_d = 1,x_{d-1} = 1,\ldots,x_{d+1-k} = 1\}.
\]
Thus, we have that $[-1,1]^d_0 = [-1,1]^d$.
We will use the following fact whose proof can be found in \cite{BJM16}.  
%---THEOREM: Cubes---
\begin{theorem}
\cite[Theorem 5.1]{BJM16}
\label{thm: cubes}
For $0\leq k\leq d$,
\[
h^\ast\left([-1,1]_k^d;x\right) = B_{k+1}^{d+1}(x).
\]
\end{theorem}
Moreover, one can verify that 
\begin{equation}
\label{eqn: type-B ehrhart}
i\left([-1,1]^d_k;x\right) = \sum_{[k]\subseteq S\subseteq[d]}(2x)^{|S|} = (2x)^k(2x+1)^{d-k}.
\end{equation}
%Using this fact, it can further be recovered by way of \cite[Lemma 2.1]{SV13} that
%\begin{equation}
%\label{eqn: type-B ehrhart}
%i\left([-1,1]^d_k;x\right) = \sum_{[k]\subseteq S\subseteq[d]}(2x)^{|S|} = (2x)^k(2x+1)^{d-k}.
%\end{equation}
Notice that the final equality implies that $B_{k+1}^{d+1}(x) = A_{d,2}^k(x)$, a special case of the polynomials defined in equation~\eqref{eqn: partial eulerian}.
These observations, together with the following lemma on the map $\phi_k$, will allow us to prove the desired interlacing properties for centrally-symmetric lattice zonotopes.  
%---LEMMA: T_k map for B polynomials----
\begin{lemma}
\label{lem: Tk map for B polynomials}
Let $m,d,$ and $k$ be nonnegative integers, and let $\B_{d,m}$ denote the nonnegative span of all polynomials of the form
$
x^j(mx+1)^{d-j}
$
for $j = 0,1,\ldots,d$.  
Then 
\[
\phi_k(\B_{d,m}) \subseteq\B_{d,m}.
\]
\end{lemma}

\begin{proof}
Since $\phi_k$ is a linear operator satisfying the recursion~\eqref{eqn: recursion}, then it suffices to prove that $\phi_k((mt+1)^n)\in\B_{d,m}$ for any $n\leq d$ and any $k$.  
However, this fact follows from the observation that
\begin{align*}
T_k((mx+1)^n) 
&= \sum_{i=0}^k(-1)^{k-i}{k\choose i}(m(i+1)x+1)^n,	\\
&= \sum_{i=0}^k(-1)^{k-i}{k\choose i}((mix)+(mx+1))^n,	\\
&= \sum_{i=0}^k(-1)^{k-i}{k\choose i}\sum_{\ell=0}^n{n\choose\ell}(mi)^\ell x^\ell(mx+1)^{n-\ell},	\\
&= \sum_{\ell=0}^n{n\choose\ell}m^\ell\left(\sum_{i=0}^k(-1)^{k-i}{k\choose i}i^\ell\right)x^\ell(mx+1)^{n-\ell},\\
&= k!\sum_{\ell=0}^n{n\choose\ell}m^\ell S(\ell,k)x^\ell(mx+1)^{n-\ell}.
\end{align*}
This completes the proof.
\end{proof}

%----THEOREM: Centrally Symmetric Zonotopes-----
\begin{theorem}
\label{thm: centrally symmetric zonotopes}
Let $\Zone$ be a $d$-dimensional centrally symmetric lattice zonotope and let $\varphi$ be a combinatorially-positive valuation.  
Then 
\[
\I_d\left(h^\varphi(\Zone;x)\right)\prec h^\varphi(\Zone;x).
\]  
In particular, $h^\varphi(\Zone;x)$ is alternatingly increasing. 
\end{theorem}

\begin{proof}
Recall from equation~\eqref{eqn: colored E} that $E_{2,k}^d(x)=\E((2x)^k(2x+1)^{d-k})$ for $0\leq k\leq d$.
Equation~\eqref{eqn: type-B ehrhart} then implies that for $0\leq k\leq d$
\[
E_{2,k}^d(x) = (1+x)^dB_{k+1}^{d+1}\left(\frac{x}{1+x}\right).
\]
Thus, by Theorem~\ref{thm: colored eulerian polynomials}, we have that any polynomial $p$ of the form 
\[
p = \sum_{k=0}^dc_kB_{k+1}^{d+1}
\]
satisfies $\I_d(p)\prec p$.
By \cite[Proposition 4.10]{BJM16} together with \cite[Theorem 5.3]{BJM16} we know that for any $d$-dimensional centrally-symmetric lattice zonotope $\Zone$, the Ehrhart $h^\ast$-polynomial is expressible as 
\[
h^\ast(\Zone;x) = \sum_{k=0}^dc_kB_{k+1}^{d+1},
\]
for some $c_k\geq 0$.  
If $f^\ast(\Zone;x)$ is the $f$-polynomial of $h^\ast(\Zone;x)$ as defined in equation~\eqref{eqn: f-polynomial}, it then follows that $f^\ast(\Zone;x) = \sum_{k=0}^dc_kE_{2,k}^d\in\A_d$, and that $\Zone$ has Ehrhart polynomial
\begin{equation}
\label{eqn: ehrhart polynomial for Type-B}
i(\Zone;x) = \sum_{k=0}^dc_k(2x)^k(2x+1)^{d-k}.
\end{equation}
Given a combinatorially-positive valuation $\varphi$, let $f^\varphi(\Zone;x)$ denote the $f$-polynomial associated to $h^\varphi(\Zone;x)$ via equation~\eqref{eqn: f-polynomial}.  
Since equation~\eqref{eqn: ehrhart polynomial for Type-B} implies that the Ehrhart polynomial $i(\Zone;x)$ lies in the cone $\B_{d,2}$, then equation~\eqref{eqn: combinatorially-positive} together with Lemma~\ref{lem: Tk map for B polynomials} implies that $f^\varphi(\Zone;x)\in\A_d$ for any combinatorially-positive valuation $\varphi$.  
It follows from Theorem~\ref{thm: interlacing TFAE} that $\I_d\left(h^\varphi(\Zone;x)\right)\prec h^\varphi(\Zone;x)$.
\end{proof}

%---SECTION: h-polynomials of Cohen-Macaulay Simplicial Complexes------------
\section{Doubly Cohen-Macaulay Level Simplicial Complexes}
\label{sec: h-polynomials of cohen-macaulay simplicial complexes}
We now use the techniques presented in Section~\ref{sec: the real-rootedness properties of symmetric decompositions} to strengthen some known results \cite{BW08} on the real-rootedness of $h$-polynomials of barycentric subdivisions of Boolean complexes when the complex is also doubly Cohen-Macaulay and level.  
Given a CW-complex $\Delta$, we define a partial order on its open cells $\preceq_\Delta$ such that for two open cells $\sigma,\sigma^\prime\in\Delta$, we have $\sigma\preceq_\Delta\sigma^\prime$ if and only if $\sigma$ is contained in the closure of $\sigma^\prime$ in $\Delta$.  
When $\Delta$ is regular, the partial order $\preceq_\Delta$ completely describes the topology of $\Delta$ up to homeomorphism.  
Let $P_\Delta$ denote the partially ordered set on the cells of $\Delta$ given by $\preceq_\Delta$.  
A regular CW-complex is called a \emph{Boolean complex} if for any $A\in\Delta$ the interval $[\emptyset,A]$ in $P_\Delta$ is a Boolean lattice; i.e., the lattice of subsets of a set.  
Notice that any simplicial complex is an example of a Boolean complex.  
For a Boolean complex $\Delta$, a cell $\sigma\in\Delta$ is called a \emph{face} of $\Delta$, and $\dim(\sigma)$ will denote its dimension.  
The dimension of $\Delta$, denoted $\dim(\Delta)$, is the maximum dimension of any face $\sigma\in\Delta$.  
Suppose that $\Delta$ is a Boolean complex with $d-1:=\dim(\Delta)$. 
For $-1\leq i\leq d-1$, let $f_k(\Delta)$ denote the number of faces of $\Delta$ of dimension $k$.  
Here, we assume that the dimension of $\emptyset$ is equal to $-1$.  
The polynomial 
\[
f(\Delta;x) = \sum_{i=-1}^{d-1}f_i(\Delta)x^{i+1},
\]
is called the $f$-\emph{polynomial} of $\Delta$.
We further let $h(\Delta;x) = \sum_{i=0}^dh_i(\Delta)x^i$ denote the $h$-polynomial of $\Delta$ as defined by the relationship in equation~\eqref{eqn: f-polynomial}; i.e.,
\[
f(\Delta;x) = (1+x)^dh\left(\Delta;\frac{x}{1+x}\right).
\]
The \emph{barycentric subdivision} of a Boolean cell complex $\Delta$ is the simplicial complex $\sd(\Delta)$ that is abstractly defined to have $i$-dimensional faces given by the strictly increasing flags
\[
\sigma_0\prec_\Delta\sigma_1\prec_\Delta\cdots\prec_\Delta\sigma_i,
\]
of nonempty faces $\sigma_0,\ldots,\sigma_i\in\Delta\backslash\{\emptyset\}$.  
In \cite{BW08}, the authors proved that if $\Delta$ is a Boolean cell complex with $h_i(\Delta)\geq0$ for all $0\leq i\leq d$, then $h(\sd(\Delta);x)$ is a real-rooted polynomial.  
A well-studied family of Boolean cell complexes known to satisfy the nonnegativity conditions $h_i(\Delta)\geq0$ for all $0\leq i\leq d$ are the Cohen-Macaulay simplicial complexes.  

Let $\Delta$ be a Cohen-Macaulay simplicial complex of dimension $d-1$ with $h$-polynomial 
\[
h(\Delta;x) = h_0(\Delta)+h_1(\Delta)x+\cdots+h_s(\Delta)x^s,
\]
such that $h_s(\Delta)\neq0$ and $h_{s+1}(\Delta) = \cdots = h_d(\Delta) = 0$.  
We say $\Delta$ is \emph{level} if the canonical module of its face ring is generated by $h_s$ elements (see \cite{S07} for all unfamiliar definitions).  
Examples of level simplicial complexes include the Gorenstein complexes, which correspond to the case where $h_s(\Delta) = 1$.  
It is well-known that if $\Delta$ is Gorenstein then $h(\Delta;x)$ is symmetric with respect to degree $s$.  
More generally, level complexes have generically Gorenstein face rings, which implies that 
\begin{equation}
\label{eqn: generically gorenstein}
h_0(\Delta)+h_1(\Delta)+\cdots+h_i(\Delta)\leq h_s(\Delta)+h_{s-1}(\Delta)+\cdots+h_{s-i}(\Delta),
\end{equation}
for all $0\leq i\leq s$ \cite[Proposition 3.3]{S07}.  
A Cohen-Macaulay simplicial complex $\Delta$ is called \emph{doubly Cohen-Macaulay} (or $2$-\emph{Cohen-Macaulay}) if for every vertex $v\in\Delta$ the subcomplex $\Delta\setminus v:=\{F\in\Delta : v\notin F\}$ is Cohen-Macaulay of the same dimension.  
It is noted in \cite[Chapter III.3]{S07} that a $(d-1)$-dimensional level simplicial complex $\Delta$ is doubly Cohen-Macaulay if and only if $h(\Delta;x)$ has degree $d$.  
In this case, the inequalities~\eqref{eqn: generically gorenstein} coincide with those in~\eqref{eqn: hibi}.  
%---THEOREM: Cohen Macaulay Complexes----
\begin{theorem}
\label{thm: cohen-macaulay complexes}
Let $\Delta$ be a doubly Cohen-Macaulay level simplicial complex.  
Then $h(\sd(\Delta);x)$ has a real-rooted $\I$-decomposition and, therefore, is alternatingly increasing.
\end{theorem}

\begin{proof}
In \cite[Equation 3.5]{BW08} it is shown that for a $(d-1)$-dimensional Boolean cell complex 
\[
\sum_{m\geq0}\left(\sum_{k=0}^dh_k(\Delta)m^k(m+1)^{d-k}\right)x^m = \frac{h(\sd(\Delta);x)}{(1-x)^{d+1}}.
\]
It follows from the above discussion that if $\Delta$ is further a doubly Cohen-Macaulay level simplicial complex then 
\[
h_0(\Delta)+\cdots+h_i(\Delta)\leq h_d(\Delta)+\cdots+h_{d-i}(\Delta),
\]
for all $0\leq i\leq d$.  
Thus, by Theorem~\ref{thm: real-rooted I-decomposition}, we conclude that $h(\sd(\Delta);x)$ has a real-rooted $\I$-decomposition, and therefore is alternatingly increasing.  
\end{proof}

%---REMARK: More General Complexes----
\begin{remark}
\label{rmk: more general complexes}
More generally, if $\Delta$ is a Boolean complex whose $h$-polynomial has nonnegative coefficients satisfying the inequalities~\eqref{eqn: hibi} of Theorem~\ref{thm: real-rooted I-decomposition} then the $h$-polynomial of its barycentric subdivision will have a real-rooted $\I$-decomposition and consequently be alternatingly increasing.  
It would therefore be interesting to know what Boolean cell complexes satisfy the inequalities~\eqref{eqn: hibi}.  
\end{remark}

We also note that the family of doubly Cohen-Macaulay level simplicial complexes contains well-studied families of simplicial complexes arising in combinatorics. 
Given a matroid $M = ([n];\mathfrak{I})$ of rank $d$ with ground set $[n]$ and independent sets $\mathfrak{I}$, the \emph{matroid complex} $\Delta(M)$ is the simplicial complex whose $i$-dimensional faces correspond to the independent subsets of $M$ with cardinality $i+1$.
In \cite[Theorem 3.4]{S07} it is noted that a matroid complex is always level, and in particular, they satisfy the inequalities needed in Theorem~\ref{thm: real-rooted I-decomposition}.  
%In fact, the $h$-polynomial of any matroid complex is known to satisfy the inequalities $h_i\leq h_{d-i}$ \textcolor{red}{Citation needed} \cite{XXX}.
Let $\mathcal{B}$ denote the collection of all bases of $M$.  
An element $i\in[n]$ is said to be a \emph{coloop} of $M$ if $i\in B$ for all $B\in\mathcal{B}$.  
The matroid $M$ is called \emph{coloop-free} if it contains no coloop.  
A matroid complex $\Delta(M)$ is doubly Cohen-Macaulay if and only if $M$ is coloop-free \cite{S07}. 
In particular, we have the following corollary to Theorem~\ref{thm: cohen-macaulay complexes}.
%---COROLLARY: Coloop-Free----
\begin{corollary}
\label{cor: coloop-free}
If $M$ is a coloop-free matroid then $h(\sd(\Delta(M));x)$ has a real-rooted $\I$-decomposition.
In particular, $h(\sd(\Delta(M));x)$ is alternatingly increasing.
\end{corollary}

%---ACKNOWLEDGEMENTS------
\smallskip

\noindent
{\bf Acknowledgements}.
P. Br\"and\'en is a Wallenberg Academy Fellow supported by the Knut and Alice Wallenberg Foundation, and Vetenskapsr\aa det. 
L. Solus is supported by a United States NSF Mathematical Sciences Postdoctoral Research Fellowship (DMS - 1606407).

\end{document}